\documentclass[12pt,a4paper]{article} 

\usepackage{graphicx}
\usepackage{amsfonts} 
\usepackage{amsmath} 
\usepackage{amssymb}
\usepackage{amsthm}
\usepackage[ansinew]{inputenc}
\usepackage{latexsym}
\usepackage{tabularx}
\usepackage{color}
\usepackage{enumerate}
\usepackage[active]{srcltx}
\allowdisplaybreaks
\newtheorem{thm}{Theorem}

\newtheorem{lem}{Lemma}

\newtheorem{rem}{Remark}

\newtheorem{examp}{Example}

\newcommand{\abs}[1]{\left\vert#1\right\vert}

\newcommand{\eps}{\varepsilon}

\newcommand{\mmax}{{m_{\rm max}}}
\newcommand{\mmin}{{m_{\rm min}}}
\newcommand{\APP}{{\rm APP}}

\newcommand{\bsa}{\boldsymbol{a}}
\newcommand{\bsn}{\boldsymbol{n}}

\newcommand{\bsk}{\boldsymbol{k}}
\newcommand{\bsm}{\boldsymbol{m}}
\newcommand{\bsl}{\boldsymbol{l}}

\newcommand{\bsx}{\boldsymbol{x}}

\newcommand{\bsh}{\boldsymbol{h}}

\newcommand{\bsr}{\boldsymbol{r}}

\newcommand{\bsb}{\boldsymbol{b}}

\newcommand{\cO}{{\cal O}}

\newcommand{\bsy}{\boldsymbol{y}}

\newcommand{\wal}{{\rm wal}}

\newcommand{\icomp}{\mathtt{i}}

\newcommand{\rd}{\,\mathrm{d}}
\newcommand{\e}{\,\varepsilon}

\newcommand{\NN}{\mathbb{N}}
\newcommand{\ZZ}{\mathbb{Z}}
\newcommand{\RR}{\mathbb{R}}
\newcommand{\CC}{\mathbb{C}}

\newcommand{\cB}{{\cal B}}

\newcommand{\h}{{\rm Her}}

\newcommand{\EXCLUDE}[1]{}

\makeatletter
\newcommand{\rdots}{\mathinner{\mkern1mu\lower-1\p@\vbox{\kern7\p@\hbox{.}}
\mkern2mu \raise4\p@\hbox{.}\mkern2mu\raise7\p@\hbox{.}\mkern1mu}}
\makeatother

\date{}

\begin{document}

\title{Tractability of Multivariate Approximation Defined \\
over Hilbert Spaces with Exponential Weights }

\author{Christian Irrgeher\thanks{C. Irrgeher is supported
by the Austrian Science Fund (FWF): Project F5509-N26, which is a
part of the Special Research Program "Quasi-Monte Carlo Methods:
Theory and Applications".}\;, Peter Kritzer\thanks{P. Kritzer is
supported by the Austrian Science Fund (FWF): Project F5506-N26,
which is a part of the Special Research Program "Quasi-Monte Carlo
Methods: Theory and Applications".}\;, Friedrich
Pillichshammer\thanks{F. Pillichshammer is supported by the
Austrian Science Fund (FWF): Project F5509-N26, which is a part of
the Special Research Program
"Quasi-Monte Carlo Methods: Theory and Applications".},\\
Henryk Wo\'{z}niakowski\thanks{H. Wo\'zniakowski is supported by
the NSF and by the National Science Centre, Poland, based on the
decision DEC-2013/09/B/ST1/04275.}}

\maketitle

%
%


\begin{abstract}
We study multivariate approximation defined over tensor product
Hilbert spaces. The domain space is a weighted tensor product
Hilbert space with exponential weights which depend on two
sequences $\bsa=\{a_j\}_{j\in\NN}$ and $\bsb=\{b_j\}_{j\in\NN}$ of
positive numbers, and on a bounded sequence of positive integers
$\bsm=\{m_j\}_{j\in\NN}$. The sequence $\bsa$ is non-decreasing
and the sequence $\bsb$ is bounded from below by a positive
number. We find necessary and sufficient conditions on $\bsa,\bsb$
and $\bsm$ to achieve the standard and new notions of
tractability in the worst case setting.
\end{abstract}


\section{Introduction}\label{sectractability}
We approximate $s$-variate problems by algorithms
that use finitely many linear functionals.
The information complexity $n(\eps,s)$ is defined as the minimal
number of linear functionals 
which are needed
to find an approximation to within an error threshold $\eps$.

The standard notions of tractability deal with the
characterization of $s$-variate problems for which the information
complexity $n(\eps,s)$ is \emph{not} exponential in $\eps^{-1}$
and $s$. Since there are many different ways of measuring the lack
of the exponential dependence we have various notions of
tractability. For instance, weak tractability (WT) means that
$\log\,n(\eps,s)/(s+\eps^{-1})$ goes to zero as $s+\eps^{-1}$
approaches infinity, whereas quasi-polynomial tractability (QPT)
means that $n(\eps,s)$ can be bounded for all $s\in\NN$ and all
$\eps\in(0,1]$ by $C\,\exp(t\,(1+\log\,s)(1+\log\,\eps^{-1}))$ for
some $C$ and $t$ independent of both $\eps^{-1}$ and $s$.
Analogously, we have polynomial tractability (PT) if $n(\eps,s)$
can be bounded by a polynomial in $\eps^{-1}$ and $s$, and strong
polynomial tractability (SPT) if $n(\eps,s)$ can be bounded by a
polynomial in $\eps^{-1}$ for all $s$. 
These notions of tractability have been extensively studied in
many papers and the current state of the art in this field can be
found in \cite{NW08,NW10,NW12}.

The notion of WT was recently refined in \cite{SW14} by
introducing $(t_1,t_2)$-weak tractability ($(t_1,t_2)$-WT) by
assuming that $\log\,n(\eps,s)/(s^{t_1}+\eps^{-t_2})$ goes to zero
as $s+\eps^{-1}$ approaches infinity 
for some positive $t_1$ and $t_2$. 
Uniform weak tractability (UWT) was defined in~\cite{S13} by assuming
that $(t_1,t_2)$-WT holds for all $t_1,t_2\in(0,1]$. 
It is easy to check that for $t_1,t_2\in(0,1]$ 
we have the following hierarchy
$$
\mbox{SPT}\ \Rightarrow\ \mbox{PT}\ \Rightarrow\ \mbox{QPT}\
\Rightarrow\ \mbox{UWT} \ \Rightarrow \ 
(t_1,t_2)\mbox{-WT}\  \Rightarrow
\ \mbox{WT}.
$$

All these standard notions are appropriate for $s$-variate
problems for which the minimal errors are
polynomially decaying. That is, for any $n\in\NN$ we can find $n$
linear functionals and an algorithm using these $n$ linear
functionals whose error decays like $\mathcal{O}(n^{-p})$ for some
positive $p$ and with the factor in the big $\mathcal{O}$ notation
that may depend on $s$.

There is a stream of work with new notions of tractability which
is relevant for $s$-variate problems for which the
minimal errors are exponentially decaying, see
\cite{DKPW13,DLPW11,IKLP14,KPW14,KPW14a,PP14}. 
The new notions of
tractability correspond to the standard notions of tractability
but for the pair $(s,1+\log\,\eps^{-1})$ instead of the pair
$(s,\eps^{-1})$. For instance the new notion of strong polynomial
tractability means that we can bound $n(\eps,s)$ by a polynomial
in $1+\log\,\eps^{-1}$ for all $s\in\NN$. Obviously, the new
notions of tractability are more demanding than the standard ones.
To distinguish them from the standard notions we add the prefix EC
(exponential convergence) and we have EC-WT, EC-UWT,
EC-$(t_1,t_2)$-WT,
EC-QPT, EC-PT, and EC-SPT. For $t_1,t_2\in(0,1]$, we obviously
have 
$$
\mbox{EC-SPT}\ \Rightarrow\ \mbox{EC-PT}\ \Rightarrow\ \mbox{EC-QPT}\
\Rightarrow\ \mbox{EC-UWT} \ \Rightarrow \ 
\mbox{EC-}(t_1,t_2)\mbox{-WT}\  \Rightarrow
\ \mbox{EC-WT}.
$$

We study $(t_1,t_2)$-WT and  EC-$(t_1,t_2)$-WT for
general positive $t_1$ and $t_2$, i.e., dropping the assumption
that they are from $(0,1]$. Obviously, if $t_1>1$ we do not have
an exponential dependence on $s^{\,t_1}$ but we may have the
exponential dependence on $s^{\,\tau}$ for $\tau<t_1$. For
$\tau=1$, we may have an exponential dependence on $s$ which is
usually called the curse of of dimensionality. Nevertheless, the
parameters $t_1$ and $t_2$ control the level of exponential
behaviour with respect to $s$ and $\eps^{-1}$, and it seems to be
an interesting problem to find the minimal, say, $t_1$ for which
we have $(t_1,t_2)$-WT or EC-$(t_1,t_2)$-WT.

In this paper we study all these standard and new notions of
tractability. This is done for general multivariate approximation
defined over tensor product Hilbert spaces in the worst case
setting. The construction of our problem is roughly as follows.
For $s=1$, we take a separable Hilbert space $H$ of infinite
dimension with an orthonormal basis $\{e_k\}_{k \in \NN_0}$, where
$\NN_0=\{0,1,2,\ldots\}$, and inner product $\langle \cdot , \cdot
\rangle_H$. In general, we do not assume that $H$ is a space of
functions or that it is a reproducing kernel Hilbert space.
Therefore we can only consider
linear functionals as information used
by algorithms.

{}From the space $H$, we construct a weighted Hilbert space in the
following way. For given positive numbers 
$a,b,\omega$ with $\omega\in(0,1)$,
and a bounded sequence $\bsm=\{m_k\}_{k \in \NN_0}$ of positive
integers, 
the Hilbert space $H_{a,b}$ is a subspace of $H$
for which $f\in H_{a,b}$ iff
$$
\|f\|_{H_{a,b}}:=\left(
\sum_{j=0}^{m_0-1}|\left<f,e_j\right>_H|^2\,+
\sum_{k=1}^\infty \omega^{-ak^b}\,
\sum_{j=0}^{m_k-1}|\left<f,e_{m_0+\cdots+m_{k-1}+j}\right>_H|^2\,
\right)^{1/2}
<\infty.
$$
Note that $\omega^{-ak^b}$ goes exponentially fast to infinity
with $k$. Therefore, $\|f\|_{H_{a,b}}<\infty$ means that the sum
of $|\left<f,e_{m_0+\cdots+m_{k-1}+j}\right>_H|^2$ 
for  $j=0,1,\dots,m_k-1$
must decay exponentially fast with $k$.

The univariate approximation problem $\APP_1:H_{a,b}\to H$ is
defined as the embedding operator $\APP_1f=f$. The $s$-variate
approximation problem
$$\APP_s:H_{s,\bsa,\bsb}\to H_s:=\bigotimes_{j=1}^sH$$
is the embedding operator $\APP_sf=f$, where
$$
H_{s,\bsa,\bsb}:=H_{a_1,b_1}\otimes H_{a_2,b_2}\otimes \cdots
\otimes H_{a_s,b_s}
$$
is the $s$-fold tensor product of the weighted spaces
$H_{a_j,b_j}$. Here $\bsa=\{a_j\}_{j\in\NN}$ and
$\bsb=\{b_j\}_{j\in\NN}$. We assume that $a_1>0$, the $a_j$'s are
nondecreasing, and  $\inf_jb_j>0$.

The space $H_{s,\bsa,\bsb}$ is a subset of $H_s$ with
exponentially decaying coefficients in the basis of~$H_s$. The
speed of the decay depends on the parameters $\bsa,\bsb,\bsm$ and
$\omega$ of the problem.

Special instances of the spaces $H_{s,\bsa,\bsb}$ are weighted Hermite
and Korobov spaces which were already analyzed in the papers
mentioned before. In fact, similarity in the analysis of weighted
Hermite and Korobov spaces was an indication that more general
weighted spaces can be also analyzed and it was the beginning of
this paper. Other 
special instances of $H_{s,\bsa,\bsb}$
are $\ell_2$, cosine and Walsh spaces which have
not been analyzed in this context before.

The weighted Hermite and Korobov spaces consist of \emph{analytic}
functions. This property is \emph{not} shared, in general, for
spaces $H_{s,\bsa,\bsb}$. As for the univariate case, the space
$H_{s,\bsa,\bsb}$ does not have to be a space of functions. But
even if we assume that $H_{s,\bsa,\bsb}$ is a space of functions
then the functions $e_k$ do not have to be analytic or even
smooth (for example this is the case for the Walsh space). 
It turns out that analyticity or smoothness of the functions
$e_k$ is irrelevant. Instead, the exponential decay of the
coefficients in the basis of $H_s$ is important. That is why the
results for the space $H_{s,\bsa,\bsb}$ are similar to the results
for the weighted Hermite and Korobov spaces. 

We now briefly summarize the main results obtained in this paper.
We first study when exponential convergence (EXP) and uniform
exponential convergence (UEXP) hold. EXP holds if there is
$q\in(0,1)$ such that for all $s\in\NN$ we can find positive
$C_s,M_s,p_s$ for which the $n$th minimal worst case error for
approximating $\APP_s$, see Section~\ref{secapprox}, is bounded by
$$
C_s\,q^{\,(n/M_s)^{p_s}}\ \ \ \ \mbox{for all}\ \ n\in\NN.
$$
The supremum of such $p_s$ is called the exponent of EXP and
denoted by $p_s^*$. UEXP holds if we can take $p_s=p>0$ for all
$s\in\NN$, and the supremum of such $p$ is called the exponent of
UEXP and denoted by $p^*$.

We prove that EXP holds always with no extra conditions on the
parameters $\bsa,\bsb,\bsm,\omega$, and
$p_s^*=1/\sum_{j=1}^sb_j^{-1}$, whereas UEXP holds iff
$B:=\sum_{j=1}^\infty b_j^{-1}<\infty$ and then $p^*=1/B$. Hence,
UEXP only requires that $b_j^{-1}$'s are summable and there are no
extra conditions on the rest of the parameters.

We now turn to tractability. We obtain necessary and sufficient
conditions on 
standard and new notions of tractability in
terms of the parameters $\bsa,\bsb,\bsm$ and
$\omega$ of the problems. 
Such conditions were not known before even for weighted
Hermite or Korobov spaces. More precisely, UWT, QPT,  EC-UWT, EC-QPT 
as well as $(t_1,t_2)$-WT and EC-$(t_1,t_2)$-WT were not studied before
for the weighted 
Korobov spaces, and approximation has not been 
studied at all for Hermite spaces before.

To stress that we approximate $\APP_s$, we denote the information
complexity $n(\eps,s)$ 
by $n(\eps,\APP_s)$, see again Section~\ref{secapprox}.
In this paper we
present specific lower and upper bounds 
on $n(\eps,\APP_s)$ from which we conclude
various notions of standard and new tractability. We also present
estimates of the tractability exponents. They are defined as the
infimum of $t$ for QPT and EC-QPT, or the infimum of the degree of 
polynomials in $\eps^{-1}$ for SPT and in $1+\log\,\eps^{-1}$ for
EC-SPT  which bound the information complexity $n(\eps,\APP_s)$.
We usually do not have the exact values of these exponents but
only lower and  upper bounds. It would be of interest to improve
these bounds. In this section, we only mention when various tractability
notions hold. We prove:

\begin{itemize}
\item $(t_1,t_2)$-WT holds for the parameters
$\bsa,\bsb,\bsm$ and $\omega$ iff $t_1>1$ or $m_0=1$.
\item EC-$(t_1,t_2)$-WT holds for the parameters
$\bsa,\bsb,\bsm$ and $\omega$ iff 
$t_1>1$, \newline or $t_2>1$ and $m_0=1$.
\item WT holds iff $m_0=1$, whereas EC-WT holds iff
$$
m_0=1\ \ \mbox{ and } \ \lim_{j\to\infty}\, a_j=\infty.
$$
\item UWT holds iff $m_0=1$, whereas EC-UWT holds iff 
$$
m_0=1\ \ \mbox{ and } \ \lim_{j\to\infty}\ \frac{\log\, a_j}{\log\,j}=\infty.
$$ 
\item QPT holds iff $m_0=1$, whereas EC-QPT holds iff
$$
m_0=1, \ \ \sup_{s\in\NN}\
\frac{\sum_{j=1}^sb_j^{-1}}{1+\log\,s}<\infty,\ \ \mbox{ and } \
\liminf_{j\to\infty}\ \frac{(1+\log\,j)\,\log\,a_j}{j}>0.
$$
\item PT holds iff SPT holds iff\footnote{Under a simplifying
assumption that the limit of $a_j/\log\,j$ exists.}
$$
m_0=1\ \ \mbox{ and } \ \lim_{j\to\infty}\ \frac{a_j}{\log j}>0.
$$
\item EC-PT holds iff EC-SPT holds iff
$$
m_0=1,\ \ \ \sum_{j=1}^{\infty}b_j^{-1}<\infty,\ \ \mbox{ and } \
\liminf_{j\to\infty}\ \frac{\log\,a_j}{j}>0.
$$
\end{itemize}

We would like to mention that the results for EC-UWT and EC-QPT for 
one particular example of $H_{s,a,b}$ (of Korobov type -- see 
Example \ref{korobov}) were independently shown by Guiqiao Xu
(private communication).  

Observe that for $m_0>1$, only $(t_1,t_2)$-WT and
EC-$(t_1,t_2)$-WT with $t_1>1$ hold. The reason is that
$$
n(\eps,\APP_s)\ge m_0^s\ \ \ \ \mbox{for all}\ \ \eps\in(0,1),
$$
and we have the curse of dimensionality. This also shows that the
condition $t_1>1$ is sharp. So we have to assume that $m_0=1$ to
obtain other notions of tractability in terms of the conditions on
$\bsa$ and $\bsb$. Interestingly enough there are no conditions on
$m_k$ for $k>0$ and on $\omega$. However, the exponents of
tractability as well as constants
depend on $m_k$ for $k>0$ and on~$\omega$. 
We illustrate the necessary and sufficient conditions on various
notions of tractability for $m_0=1$ and for
$$
a_j=j^{\,v_1}\,\exp(v_2\,j)\ \ \ \ \mbox{and}\ \ \ \ b_j=j^{\,v_3} \ \ \mbox{ for } j \ge 1
$$
for some non-negative $v_1,v_2$ and $v_3$. Then
\begin{itemize}
\item EXP, $(t_1,t_2)$-WT,
EC-$(t_1,t_2)$-WT with $t_1>1$, WT
and QPT hold for all $v_1,v_2,v_3$, \item UEXP holds iff $v_3>1$,
\item EC-WT, PT and SPT hold iff $v_1^2+v_2^2>0$, \item EC-PT and
EC-SPT hold iff $v_2>0$ and $v_3>1$.
\end{itemize}

The remaining sections of this paper are structured 
in the following way. We provide 
detailed information on the Hilbert spaces which are studied in the paper
in Section~\ref{sec:WHS}.
We outline the setting of the approximation problem 
in Section~\ref{secapprox}. The results on 
exponential and uniform exponential convergence 
are shown in Section~\ref{secexp}.
In Section~\ref{sectract} we prove 
the results on the various notions of tractability. A table which summarizes all conditions
is presented in Section~\ref{sec_sum}.


\section{Weighted Hilbert Spaces}\label{sec:WHS}
Let $H$ be a separable Hilbert space over the real or complex
field. To omit special cases, we also assume that $H$ has infinite
dimension. Let $\{e_k\}_{k\in\NN_0}$ be its orthonormal basis,
$\left<e_k,e_j\right>_{H}=\delta_{k,j}$ for all $k,j\in\NN_0$.
Hence, for $f\in H$ one has
$$
f=\sum_{k=0}^{\infty} \left<f,e_k\right>_H\,e_k\ \ \ \ \mbox{with}\ \
\ \ \sum_{k=0}^{\infty}|\left<f,e_k\right>_H|^2<\infty.
$$
For $s\in\NN:=\{1,2,\dots\}$, by $H_s=H\otimes
H\otimes\cdots\otimes H$ we mean the $s$-fold tensor product of
$H$. For $\bsk=[k_1,k_2,\dots,k_s]\in\NN_0^s$, let
$e_{\bsk}=e_{k_1}\otimes e_{k_2}\otimes\cdots\otimes e_{k_s}$.
Clearly, $\{e_{\bsk}\}_{\bsk\in\NN_0^s}$ is an orthonormal basis
of $H_s$ and for $f\in H_s$ one has
$$
f=\sum_{\bsk\in\NN_0^s} \left<f,e_{\bsk}\right>_{H_s}\,e_{\bsk}\ \
\ \ \mbox{with}\ \ \ \ \sum_{\bsk\in\NN_0^s}|\left<f,e_{\bsk}
\right>_{H_s}|^2<\infty.
$$

We now define a weighted Hilbert space which will depend on a
number of parameters. Some of these parameters will be fixed while
others will be varying. The fixed parameters are: a number
$\omega\in(0,1)$ and a bounded sequence $\bsm=\{m_k\}_{k\in\NN_0}$
of positive integers. 
With the sequence
$\bsm$ we associate a sequence $\bsr=\{r_k\}_{k\in\NN_0}$ given by
\begin{eqnarray*}
r_0&=&0,\\
r_k&=&m_0+m_1+\dots+m_{k-1}\ \ \ \ \mbox{for all}\ \ k\in\NN.
\end{eqnarray*}
Clearly, $r_{k+1}=r_k+m_k\ge r_k+1$. Furthermore,
$$
\NN_0=\bigcup_{k=0}^\infty \{r_k,r_k+1,\ldots,r_{k+1}-1\}
$$
and the sets $\{r_k,r_k+1,\ldots,r_{k+1}-1\}$ are disjoint.

The varying parameters are positive real numbers $a$ and $b$. The
weighted Hilbert space will be therefore denoted by $H_{a,b}$ and
is defined as
$$
H_{a,b}=\bigg\{\,f\in H\ : \ \|f\|_{H_{a,b}}:=\left(
\sum_{k=0}^\infty \omega^{-ak^b}\,\sum_{j=r_k}^{r_{k+1}-1}
|\left<f,e_j\right>_H|^2\right)^{1/2}<\infty\ \bigg\}.
$$
As an example, consider $m_k\equiv 1$. Then $r_k=k$ and
$$
\|f\|_{H_{a,b}}:=\left( \sum_{k=0}^\infty \omega^{-ak^b}\,
|\left<f,e_k\right>_H|^2\right)^{1/2}.
$$
For a general $\bsm$, note that $\omega^{-ak^b}$ goes
exponentially fast to infinity with $k$. Therefore
$\|f\|_{H_{a,b}}<\infty$ means that
$\sum_{j=r_k}^{r_{k+1}-1}|\left<f,e_j\right>_H|^2$ must decay
exponentially fast to zero as $k$ goes to infinity.

The inner product in $H_{a,b}$ is given for $f,g\in H_{a,b}$ by
$$
\left<f,g\right>_{H_{a,b}}= \sum_{k=0}^\infty
\omega^{-ak^b}\,\sum_{j=r_k}^{r_{k+1}-1} \left<f,e_j\right>_H
\overline{\left<g,e_j\right>_H}.
$$
Since $\omega^{-ak^b}\ge1$, we have
\begin{equation}\label{1}
\|f\|_H\le\|f\|_{H_{a,b}}\ \ \ \ \mbox{for all}\ \ f\in H_{a,b}.
\end{equation}

We now find an orthonormal basis $\{e_{n,a,b}\}_{n\in\NN_0}$ of
$H_{a,b}$. For $n\in\NN_0$, there is
a unique $k=k(n)$ such
that $n\in\{r_{k(n)},r_{k(n)}+1,\ldots,r_{k(n)+1}-1\}$. Then we
set
$$
e_{n,a,b}=\omega^{a[k(n)]^b/2}\,e_n.
$$

We now verify that the sequence $\{e_{n,a,b}\}_{n\in\NN_0}$ is
orthonormal in $H_{a,b}$. Indeed, take $n_1,n_2\in\NN_0$. Then
\begin{eqnarray*}
\left<e_{n_1,a,b},e_{n_2,a,b}\right>_{H_{a,b}}&=&
\sum_{k=0}^\infty\omega^{-ak^b}\,\sum_{j=r_k}^{r_{k+1}-1}
\left<e_{n_1,a,b},e_j\right>_H\,\overline{\left<e_{n_2,a,b},e_j\right>_H}\\
&=& \sum_{k=0}^\infty\omega^{-ak^b+a[k(n_1)]^b/2+a[k(n_2)]^b/2}\,
\sum_{j=r_k}^{r_{k+1}-1}
\left<e_{n_1},e_j\right>_H\,\overline{\left<e_{n_2},e_j\right>_H}.
 \end{eqnarray*}
Suppose that $n_1\not=n_2$. Then the last sum over $j$ is zero for
all $k\in\NN_0$ due to the orthonormality of $\{e_j\}_{j\in
\NN_0}$. Suppose now that $n_1=n_2$. Then  the only non-zero term
is for $k=k(n_1)$ and $j=n_1$, so that the sum is~$1$. Hence,
$\left<
  e_{n_1,a,b},e_{n_2,a,b}\right>_{H_{a,b}}=\delta_{n_1,n_2}$. Finally, 
note that $H_{a,b}\subseteq H=\mbox{span}(e_1,e_2,\dots)=
\mbox{span}(e_{1,a,b},e_{2,a,b},\dots)$, which means that
$\{e_{n,a,b}\}_{n\in\NN_0}$ is an orthonormal basis of $H_{a,b}$,
as claimed.

The norm in $H_{a,b}$ can now also be written as $$\|f\|_{H_{a,b}}=\left(\sum_{n=0}^{\infty}|\langle f, e_{n,a,b}\rangle_{H_{a,b}}|^2\right)^{1/2}.$$

We remark that $k(n)=0$ for $n\in\{0,1,\ldots,m_0-1\}$ and therefore,
$e_{n,a,b}=e_n$ and
$$
\|e_{n,a,b}\|_{H_{a,b}}=\|e_n\|_H=1\ \ \ \ \mbox{for all}\ \ n\in
\{0,1,\ldots,m_0-1\}.
$$
The last equality holds for $m_0$ elements, and $m_0\ge1$. This
and \eqref{1} imply
\begin{equation}\label{2}
\sup_{\|f\|_{H_{a,b}}\le 1}\,\|f\|_H=1.
\end{equation}

Similarly as for the space $H_s$, we take the $s$-fold tensor
products of the weighted space~$H_{a_j,b_j}$ with possibly
different $a_j$ and $b_j$ such that
\begin{equation}\label{3}
0<a_1\le a_2\le\cdots\ \ \ \ \mbox{and}\ \ \ \ \ \
\inf_{j\in\NN}b_j>0.
\end{equation}
That is,
$$
H_{s,\bsa,\bsb}=H_{a_1,b_1}\otimes H_{a_2,b_2}\otimes\cdots\otimes
H_{a_s,b_s}.
$$
For $\bsn=[n_1,n_2,\dots,n_s]\in\NN_0^s$, define
$$
e_{\bsn,\bsa,\bsb}=e_{n_1,a_1,b_1}\otimes
e_{n_2,a_2,b_2}\otimes\cdots\otimes e_{n_s,a_s,b_s}.
$$
Then $\{e_{\bsn,\bsa,\bsb}\}_{\bsn\in\NN_0^s}$ is an orthonormal
basis of $H_{s,\bsa,\bsb}$ and $f\in H_{s,\bsa,\bsb}$ iff
$$
f=\sum_{\bsn\in\NN_0^s}\left<f,e_{\bsn,\bsa,\bsb}\right>_{H_{s,\bsa,\bsb}}
\,e_{\bsn,\bsa,\bsb}\ \ \ \ \mbox{with}\ \ \ \
\|f\|_{H_{s,\bsa,\bsb}}:=\left(
\sum_{\bsn\in\NN_0^s}|\left<f,e_{\bsn,\bsa,\bsb}\right>_{H_{s,\bsa,\bsb}}|^2
\right)^{1/2}<\infty.
$$
We now show that
\begin{equation}\label{4}
\|f\|_{H_s}\le\|f\|_{H_{s,\bsa,\bsb}}\ \ \ \ \mbox{for all}\ \
f\in H_{s,\bsa,\bsb}.
\end{equation}
Indeed, for $f\in H_{s,\bsa,\bsb}$ we have
$f=\sum_{\bsn\in\NN_0^s}\alpha_{\bsn}\,e_{\bsn,\bsa,\bsb}$ with
$\|f\|_{H_{s,\bsa,\bsb}}^2=\sum_{\bsn\in\NN_0^s}|\alpha_{\bsn}|^2<\infty$.
For any $n_j\in\NN_0$ there is a unique $k(n_j)\in\NN_0$ such that
$n_j\in \{r_{k(n_j)},r_{k(n_j)}+1,\ldots,r_{k(n_j)+1}-1\}$, and
$e_{n_j,a_j,b_j}=\omega^{a_j[k(n_j)]^{b_j}/2}\,e_{n_j}$. Therefore
\begin{equation}\label{5}
e_{\bsn,\bsa,\bsb}=\left(\prod_{j=1}^s\omega^{a_j[k(n_j)]^{b_j}/2}\right)
\,e_{\bsn}.
\end{equation}
We have $f=\sum_{\bsn\in\NN_0^s}\alpha_{\bsn}\left(\prod_{j=1}^s
\omega^{a_j[k(n_j)]^{b_j}/2}\right)\,e_{\bsn}$ and
$$
\|f\|_{H_s}=\left(\sum_{\bsn\in\NN_0^s}
|\alpha_{\bsn}|^2\prod_{j=1}^s\omega^{a_j[k(n_j)]^{b_j}}\right)^{1/2}
\le \left(\sum_{\bsn\in\NN_0^s}
|\alpha_{\bsn}|^2\right)^{1/2}=\|f\|_{H_{s,\bsa,\bsb}},
$$
as claimed.

For $\bsn\in\{0,1,\ldots,m_0-1\}^s$, we have $k(n_j)=0$ for
$j=1,2,\dots,s$. Therefore $e_{\bsn,\bsa,\bsb}=e_{\bsn}$ and
$$
\|e_{\bsn,\bsa,\bsb}\|_{H_{s,\bsa,\bsb}}=\|e_{\bsn}\|_{H_s}=1\ \ \
\ \mbox{for all}\ \ \bsn\in\{0,1,\ldots,m_0-1\}^s.
$$
The last equality holds for $m_0^s$ elements. This and \eqref{4}
imply
\begin{equation}\label{6}
\sup_{\|f\|_{H_{s,\bsa,\bsb}}\le1}\|f\|_{H_s}=1.
\end{equation}

Note that \eqref{5} implies that $\{e_{\bsn}\}_{\bsn\in \NN_0^s}$
is orthogonal in $H_{s,\bsa,\bsb}$ and
$$
\|e_{\bsn}\|_{H_{s,\bsa,\bsb}}=\prod_{j=1}^s\omega^{-a_j[k(n_j)]^{b_j}/2}
\ \ \ \ \mbox{for all}\ \ \bsn\in\NN_0^s.
$$
For $f\in H_s$ we have
$f=\sum_{\bsn\in\NN_0^s}\left<f,e_{\bsn}\right>_{H_s}e_{\bsn}$
with $
\sum_{\bsn\in\NN_0^s}|\left<f,e_{\bsn}\right>_{H_s}|^2<\infty$.
Such $f$ belongs to $H_{s,\bsa,\bsb}$ iff
$$
\|f\|_{H_{s,\bsa,\bsb}}=\left(\sum_{\bsn\in\NN_0^s}\
\prod_{j=1}^s\omega^{-a_j[k(n_j)]^{b_j}}
|\left<f,e_{\bsn}\right>_{H_s}|^2 \right)^{1/2}<\infty.
$$
As for the univariate case, we see that
$\prod_{j=1}^s\omega^{-a_j[k(n_j)]^{b_j}}$ goes exponentially fast
to infinity if one of the components of $\bsn$ goes to infinity.
Therefore $|\left<f,e_{\bsn}\right>_{H_s}|$ must decay
exponentially fast to zero if one of the components of $\bsn$
approaches infinity.

\begin{rem}
\end{rem}
We stress that the spaces $H$, $H_s$ and  $H_{s,\bsa,\bsb}$ do not
have to be reproducing kernel Hilbert spaces, see \cite{A50} for
general facts on reproducing kernel Hilbert spaces. Indeed, the
initial space $H$ does not have to be a function space. But if $H$
is a Hilbert space of real or complex valued functions defined on,
say, a common domain $D$, then it is well known that $H$ is a
reproducing kernel Hilbert space iff
\begin{equation}\label{7}
\sum_{k=0}^\infty|e_k(x)|^2<\infty\ \ \ \ \mbox{for all}\ \ x\in
D.
\end{equation}
If \eqref{7} holds then
\begin{align}\label{kernel}
K(x,y)=\sum_{k=0}^{\infty}e_k(x)\,\overline{e_k(y)}\ \ \ \ \mbox{for
all}\ \ x,y\in D
\end{align}
is well-defined for all $x,y\in D$, because
\begin{align*}
\vert K(x,y)\vert&=\left\vert\sum_{k=0}^{\infty}e_k(x)\overline{e_k(y)}
\right\vert\leq\left( \sum_{k=0}^{\infty}\vert e_k(x)\vert^2\sum_{k=0}^{\infty}
\vert e_k(y)\vert^2\right)^{1/2}<\infty
\end{align*}
by the Cauchy-Schwarz inequality. Moreover,
$$
f(y)=\left<f,K(\cdot,y)\right>_H\ \ \ \ \mbox{for all}\ \ f\in H \
\mbox{and}\ y\in D.
$$
holds and so \eqref{kernel} is a reproducing kernel of $H$. If \eqref{7} 
holds then $H_s$ is also a reproducing kernel Hilbert
space and its kernel is
$$
K_s(\bsx,\bsy)=\prod_{j=1}^sK(x_j,y_j)= \sum_{\bsk\in\NN_0^s}
e_{\bsk}(\bsx)\,\overline{e_{\bsk}(\bsy)}\ \ \ \ \mbox{for all}\ \
\bsx,\bsy\in D^s.
$$
Similarly, the weighted space $H_{a,b}$ is a reproducing kernel
Hilbert space iff
\begin{equation}\label{8}
\sum_{k=0}^\infty\omega^{ak^b}\,\sum_{j=r_k}^{r_{k+1}-1}|e_j(x)|^2<\infty\
\ \ \ \mbox{for all}\ \ x\in D.
\end{equation}
Clearly, the condition \eqref{8} is weaker than the condition
\eqref{7}. Hence, it may happen that $H$ is not a reproducing
kernel Hilbert space but $H_{a,b}$ is. We shall see examples of
such spaces in a moment.

If \eqref{8} holds then the reproducing kernel of $H_{a,b}$ is
$$
K_{a,b}(x,y)=\sum_{k=0}^{\infty}e_{k,a,b}(x)\overline{e_{k,a,b}(y)}=
\sum_{k=0}^\infty\omega^{ak^b}\,\sum_{j=r_k}^{r_{k+1}-1}e_j(x)
\overline{e_j(y)}\ \ \ \ \mbox{for all}\ \ x,y\in D.
$$
If \eqref{8} holds then $H_{s,\bsa,\bsb}$ is also a reproducing
kernel Hilbert space and its kernel is
$$
K_{s,\bsa,\bsb}(\bsx,\bsy)=\prod_{j=1}^sK_{a_j,b_j}(x_j,y_j)=
\sum_{\bsk\in\NN_0^s} e_{\bsk,\bsa,\bsb}(\bsx)\,
\overline{e_{\bsk,\bsa,\bsb}(\bsy)}\ \ \ \ \mbox{for all}\ \
\bsx,\bsy\in D^s.
$$

\qed

We illustrate the weighted Hilbert spaces $H_{s,\bsa,\bsb}$ by five
examples.
\begin{examp}\ {\bf Weighted $\ell_2$ Space}
\end{examp}
Let $H=\ell_2$ be the space of sequences in $\mathbb{C}$ with
finite quadratic norm, i.e., $H=\ell_2=\{f:\NN_0 \rightarrow
\mathbb{C} \, : \, \sum_{n=0}^{\infty} |f(n)|^2 < \infty\}$. Let
$e_k$ be the $k$-th canonical element $e_k(n)=\delta_{k,n}$.

Then $H_s=\{f:\NN_0^s \rightarrow \mathbb{C} \, : \, \sum_{\bsn
\in \NN_0^s} |f(\bsn)|^2 < \infty\}$. For
$\bsk=[k_1,k_2,\ldots,k_s]\in\NN_0^s$,  let
$e_{\bsk}(\bsn)=\prod_{j=1}^s e_{k_j}(n_j)$ for all
$\bsn=[n_1,n_2,\ldots,n_s]\in\NN_0^s$. The inner product in $H_s$
is $\langle f,g\rangle_{H_s}=\sum_{\bsk\in \NN_0^s} f(\bsk)
g(\bsk)$. 
Note that $\langle
e_{\bsk_1},e_{\bsk_2}\rangle_{H_s}=\delta_{\bsk_1,\bsk_2}$ 
and  
$\sum_{\bsk \in \NN_0^s} |e_{\bsk}(\bsn)|^2
=1$.  
Hence, $H_s$ is a reproducing kernel Hilbert space with
kernel function 
$$K_s(\bsl,\bsn)=\sum_{\bsk \in \NN_0^s}
e_{\bsk}(\bsl)e_{\bsk}(\bsn)=\delta_{\bsl,\bsn}\ \ \ \mbox{ for $\bsl,\bsn \in \NN_0^s$.}$$
For $H_{s,\bsa,\bsb}$, we take $m_k \equiv 1$. 
Then $r_k=k$ and
$k(n)=n$. The inner product of $H_{s,\bsa,\bsb}$ for $f,g\in
H_{s,\bsa,\bsb}$ is given by
$$
\left<f,g\right>_{H_{s,\bsa,\bsb}}=\sum_{\bsk\in\NN_0^s}
\omega^{-\sum_{j=1}^sa_jk_j^{b_j}}\,f(\bsk)\,g(\bsk).
$$
Hence $f \in H_{s,\bsa,\bsb}$ means that the $\abs{f(\bsk)}$ of $f$
decrease exponentially fast. $H_{s,\bsa,\bsb}$ is a reproducing
kernel Hilbert space with kernel
$$K_{s,\bsa,\bsb}(\bsl,\bsn)=\sum_{\bsk \in \NN_0^s}
\omega^{\sum_{j=1}^{s} a_j k_j^{b_j}}
e_{\bsk}(\bsl)e_{\bsk}(\bsn).$$
\qed

\begin{examp}\  {\bf Weighted Hermite Space}
\end{examp}

Let $\rho(x)=(2\pi)^{-1/2}\exp(-x^2/2)$ for all $x\in\RR$ be the Gaussian weight
in the real line and let $\h_k$ be the Hermite polynomial of degree $k$,
$$
\h_k(x)=\frac{(-1)^k}{\sqrt{k!}}\,\exp(x^2/2)\,\frac{{\rm
d}^k}{{\rm
    d}x^k}\, \exp(-x^2/2)\ \ \ \ \mbox{for all}\ \ x\in \RR.
$$
Now we consider the Hilbert space of real functions which are Lebesgue square-integrable
with respect to $\rho$ and have an absolutely convergent Hermite series,
i.e.,
\begin{align*}
H&=\Big\{f:\RR\rightarrow\RR\,:\,f\textnormal{ measurable}, \int_{\RR} \vert f(x)\vert^2\rho(x)\rd x<\infty,\\
&\qquad\qquad\qquad\qquad\qquad\qquad f(x)=\sum_{k\in\NN_0}\widehat{f}_k \h_k(x)\textnormal{ absolutely convergent}\Big\}
\end{align*}
with inner product $\langle f,g\rangle=\int_{\RR}f(x)g(x)\rho(x)\rd x$. 
Since it is known that $\{\h_k\}_{k\in\NN_0}$ is orthonormal, we
can take $e_k=\h_k$. Clearly, $H$ is \emph{not} a reproducing
kernel Hilbert space. Then 
\begin{align*}
H_s&=\Big\{f:\RR^s\rightarrow\RR\,:\,f\textnormal{ measurable}, \int_{\RR^s} \vert f(\bsx)\vert^2\rho_s(\bsx)\rd \bsx<\infty,\\
&\qquad\qquad\qquad\qquad\qquad\qquad f(\bsx)=\sum_{\bsk\in\NN^s_0}\widehat{f}_{\bsk} \h_{\bsk}(\bsx)\textnormal{ absolutely convergent}\Big\}
\end{align*}
with
$$
\rho_s(\bsx)=\frac1{(2\pi)^{s/2}}\,\exp\left(-\frac12\,\sum_{j=1}^sx_j^2\right)
\ \ \ \ \mbox{for all}\ \ \bsx=[x_1,x_2,\dots,x_s]\in \RR^s.
$$
For $\bsk\in\NN_0^s$, we take
$e_{\bsk}(\bsx)=\h_{\bsk}(\bsx)=\prod_{j=1}^s \h_{k_j}(x_j)$ for all
$\bsx\in\RR^s$. Then $\{e_{\bsk}\}_{\bsk\in\NN_0^s}$ is an
orthonormal basis of $H_s$. Obviously, $H_s$ is \emph{not} a
reproducing kernel Hilbert space for any~$s\in\NN$.

The weighted Hermite space $H_{s,\bsa,\bsb}$ is obtained by taking
$m_k\equiv 1$. Then $r_k=k$ and $k(n)=n$.  The inner product of
$H_{s,\bsa,\bsb}$ for $f,g\in H_{s,\bsa,\bsb}$ is given by
$$
\left<f,g\right>_{H_{s,\bsa,\bsb}}=\sum_{\bsk\in\NN_0^s}
\omega^{-\sum_{j=1}^sa_jk_j^{b_j}}\,\widehat{f}_{\bsk}\,\widehat{g}_{\bsk},
$$
where $\widehat{f}_{\bsk}$ and $\widehat{g}_{\bsk}$ denote the
$\bsk$th Hermite coefficients of $f$ and $g$,
$$
\widehat{f}_{\bsk}=\langle f,\h_{\bsk} \rangle_{L_2(\RR^s,\rho_s)}=
\int_{\RR^s}f(\bsx)\,\h_{\bsk}(\bsx)\,\rho_s(\bsx)\,{\rm
  d}\bsx\ \ \ \ \mbox{for all}\ \ \bsk\in \NN_0^s.
$$
The weighted Hermite space $H_{s,\bsa,\bsb}$ is a reproducing
kernel Hilbert space due to Cramer's bound which states that
$$
|\h_k(x)|\le (2\pi)^{1/4}\,\exp(x^2/4)\ \ \ \ 
\mbox{for all}\ \  \ x\in\RR\ \ \mbox{and}\ \ 
k\in\NN_0,
$$ 
see \cite[p. 324]{S77}. Indeed, this bound leads to
$$
\sum_{\bsk\in\NN_0^s}[e_{\bsk,\bsa,\bsb}(x)]^2 = \prod_{j=1}^s
\sum_{k=0}^{\infty}\omega^{a_jk^{b_j}}\,[\h_{k}(x_j)]^2 \le
\prod_{j=1}^s(2\pi)^{1/2}\,\exp(x_j^2/2)\,
\sum_{k=0}^\infty\omega^{a_jk^{b_j}}<\infty
$$
since the series $\sum_{k=0}^\infty\omega^{a_jk^{b_j}}<\infty$ for
all positive $a_j$ and $b_j$. The reproducing kernel of
$H_{s,\bsa,\bsb}$ is
$$
K_{s,\bsa,\bsb}(\bsx,\bsy)=
\sum_{\bsk\in\NN_0^s}\omega^{\sum_{j=1}^sa_jk^{b_j}}\,
\h_{\bsk}(\bsx)\,\h_{\bsk}(\bsy)\ \ \ \ \mbox{for all}\ \
\bsx,\bsy\in\RR^s.
$$
More information on weighted Hermite spaces can be found in \cite{IKLP14,IL}.
\qed

\begin{examp}\label{korobov} \ {\bf Weighted Korobov Space}
\end{examp}
We now take $H$ as the Hilbert space of complex-valued, square-integrable functions on $[0,1]$ with absolutely convergent Fourier series, i.e.,
\begin{align*}
H&=\Big\{f:[0,1]\rightarrow\CC\,:\,f\textnormal{ measurable}, \int_0^1 \vert f(x)\vert^2\rd x<\infty,\\
&\qquad\qquad\qquad\qquad\qquad\qquad f(x)=\sum_{k\in\ZZ}\widehat{f}_k \exp(2\pi \icomp kx)\textnormal{ absolutely convergent}\Big\}
\end{align*}
with inner product $\langle f,g\rangle=\int_0^1f(x)\overline{g(x)}\rd x$. The orthonormal basis
$\{e_k\}_{k\in\NN_0}$ of $H$  is taken as
$$
e_0(x)=1, \ \ \ e_{2k-1}(x)=\exp(2\pi\icomp k\,x),\ \ \
e_{2k}(x)=\exp(-2\pi\icomp k\,x),
$$
for $k\in\NN$ with $\icomp=\sqrt{-1}$. Then 
\begin{align*}
&H_s=\Big\{f:[0,1]^s\rightarrow\CC\,:\,f\textnormal{ measurable}, \int_{[0,1]^s}\vert f(\bsx)\vert^2\rd \bsx<\infty,\\
&\qquad\qquad\qquad\qquad\qquad\qquad f(x)=\sum_{\bsk\in\ZZ}\widehat{f}_{\bsk} \exp(2\pi \icomp \bsk\bsx)\textnormal{ absolutely convergent}\Big\}
\end{align*}
and $\{e_{\bsk}\}_{\bsk\in\NN_0^s}$  with
$$
e_{\bsk}(\bsx)=\prod_{j=1}^se_{k_j}(x_j) \ \ \ \ \mbox{for all}\ \
\bsx\in[0,1]^s
$$
as its orthonormal basis. Then $f\in H_s$ iff
$$
f(\bsx)=\sum_{\bsh\in\ZZ^s}\widehat{f}_{\bsh}\,\exp\left(
2\pi\icomp\, \bsh \cdot \bsx\right)\ \ \ \ \mbox{for all}\ \
\bsx\in[0,1]^s
$$
with $\sum_{\bsh\in\ZZ^s}|\widehat{f}_{\bsh}|^2<\infty$, where
$\bsh \cdot \bsx$ denotes the usual dot product. Here,
$\ZZ$ is the set of all integers,
$\ZZ:=\{\dots,-1,0,1,\dots\}$, and
$$
\widehat{f}_{\bsh}=\langle f,e_{\bsh} \rangle_{L_2}=
\int_{[0,1]^s}f(\bsx)\,\exp(-2\pi\icomp \, \bsh \cdot \bsx)\,{\rm
d}\bsx
$$
is the $\bsh$th Fourier coefficient. Clearly, $H_s$ is \emph{not}
a reproducing kernel Hilbert space for all~$s\in \NN$.

The weighted Korobov space $H_{s,\bsa,\bsb}$ is obtained by taking
$m_0=1$ and $m_k=2$ for all $k\in \NN$. Then $r_0=0$ and
$r_k=2k-1$ for all $k\in\NN$. The inner product of
$H_{s,\bsa,\bsb}$ for $f,g\in H_{s,\bsa,\bsb}$ is given by
$$
\left<f,g\right>_{H_{s,\bsa,\bsb}}=\sum_{\bsh\in\ZZ^s}
\omega^{-\sum_{j=1}^sa_j|h_j|^{b_j}}\, \widehat{f}_{\bsh}
\,\overline{\widehat{g}_{\bsh}}.
$$
The space $H_{s,\bsa,\bsb}$ is a reproducing kernel Hilbert space
and its reproducing kernel is
$$
K_{s,\bsa,\bsb}(\bsx,\bsy)=\sum_{\bsh\in\ZZ^s}
\omega^{\sum_{j=1}^sa_j|h_j|^{b_j}}\, \exp(2\pi\icomp \,\bsh \cdot
(\bsx-\bsy)) \ \ \ \ \mbox{for all}\ \ \bsx,\bsy\in[0,1]^s.
$$
The weighted Korobov space $H_{s,\bsa,\bsb}$ is a space of
periodic functions with period~$1$ for each variable. More 
information on
these spaces can be found in \cite{DKPW13,DLPW11,KPW14,KPW14a}.
\qed

\begin{examp} \ {\bf Weighted Cosine Space}
\end{examp}
We take $H$ as the Hilbert space of real-valued and square integrable
functions defined on $[0,1]$ with absolutely convergent cosine series,
 more precisely,
\begin{align*}
H&=\Big\{f:[0,1]\rightarrow\RR\,:\,f\textnormal{ measurable}, \int_0^1 \vert f(x)\vert^2\rd x<\infty,\\
&\qquad\qquad\qquad\qquad\qquad\qquad f(x)=\widehat{f}_0+\sum_{k\in\NN}\widehat{f}_k \sqrt{2}\cos(\pi kx)\textnormal{ absolutely convergent}\Big\}
\end{align*}
The orthonormal basis $\{e_k\}_{k\in\NN_0}$ of $H$  is then taken as
$$
e_0(x)=1, \mbox{ and } \ \ e_k(x)=\sqrt{2} \cos(\pi k x) \ \mbox{
for }\ k\in\NN.$$

Then we take, as in the previous examples, $H_s=H\otimes\cdots\otimes H$ 
and $\{e_{\bsk}\}_{\bsk\in\NN_0^s}$  with
$$
e_{\bsk}(\bsx)=\prod_{j=1}^se_{k_j}(x_j) \ \ \ \ \mbox{for all}\ \
\bsx\in[0,1]^s
$$
as its orthonormal basis. For $\bsh=[h_1,h_2,\ldots,h_s] \in
\NN_0^s$ we denote by $|\bsh|_0$ the number of indices $j \in
\{1,2,\ldots,s\}$ for which $h_j \not=0$. Then $f\in H_s$ iff
$$
f(\bsx)=\sum_{\bsh\in\NN_0^s}\widetilde{f}_{\bsh}
(\sqrt{2})^{|\bsh|_0}\left(\prod_{j=1}^s \cos(\pi h_j x_j)\right)\
\ \ \ \mbox{for all}\ \ \bsx\in[0,1]^s
$$
with $\sum_{\bsh\in\NN_0^s}|\widetilde{f}_{\bsh}|^2<\infty$. Here
$$
\widetilde{f}_{\bsh}=\langle f,e_{\bsh} \rangle_{L_2}=
\int_{[0,1]^s}f(x) (\sqrt{2})^{|\bsh|_0} \left(\prod_{j=1}^s
\cos(\pi h_j x_j)\right)  \,{\rm d}\bsx
$$
is the $\bsh$th 
cosine coefficient. Clearly, $H_s$ is
\emph{not} a reproducing kernel Hilbert space for all~$s\in \NN$.

The weighted cosine space $H_{s,\bsa,\bsb}$ is obtained by taking
$m_k\equiv 1$. Then $r_k=k$ and $k(n)=n$. The inner product of
$H_{s,\bsa,\bsb}$ for $f,g\in H_{s,\bsa,\bsb}$ is given by
$$
\left<f,g\right>_{H_{s,\bsa,\bsb}}=\sum_{\bsh\in\NN_0^s}
\omega^{-\sum_{j=1}^sa_j|h_j|^{b_j}}\, \widetilde{f}_{\bsh}
\,\widetilde{g}_{\bsh}.
$$
The space $H_{s,\bsa,\bsb}$ is a reproducing kernel Hilbert space
and its reproducing kernel is
$$
K_{s,\bsa,\bsb}(\bsx,\bsy)=\sum_{\bsh\in\NN_0^s}
\omega^{\sum_{j=1}^sa_j|h_j|^{b_j}}\, 2^{|\bsh|_0}
\left(\prod_{j=1}^s \cos(\pi h_j x_j) \cos(\pi h_j y_j)\right) \ \
\ \ \mbox{for all}\ \ \bsx,\bsy\in[0,1]^s.
$$
More information 
on cosine spaces with finite smoothness can be found
in \cite{DNP14}. \qed

\begin{examp} \ {\bf Weighted Walsh Space}
\end{examp}
Let us denote by $\wal_k$ the $k$th Walsh function in some fixed 
integer base $b\ge 2$, see for example 
\cite[Appendix~A]{DP10} for further details.

We now study the Hilbert space of complex-valued, square-integrable 
functions on $[0,1]$ with absolutely convergent Walsh series, i.e.,
\begin{align*}
H&=\Big\{f:[0,1]\rightarrow\CC\,:\,f\textnormal{ measurable}, \int_0^1 \vert f(x)\vert^2\rd x<\infty,\\
&\qquad\qquad\qquad\qquad\qquad\qquad f(x)=\sum_{k\in\NN_0}\widehat{f}_k \wal_k(x)\textnormal{ absolutely convergent}\Big\}.
\end{align*}
For this example the orthonormal basis $\{e_k\}_{k\in\NN_0}$ of $H$
is taken as 
$$
e_k(x)=\wal_k (x)\ \ \ \  \mbox{for all}\ \ k\in\NN_0.
$$
Then $H_s=L_2([0,1]^s)$ 
\begin{align*}
H_s&=\Big\{f:[0,1]^s\rightarrow\RR\,:\,f\textnormal{ measurable}, \int_{[0,1]^s} \vert f(\bsx)\vert^2\rd\bsx<\infty,\\
&\qquad\qquad\qquad\qquad\qquad\qquad f(\bsx)=\sum_{\bsk\in\NN^s_0}\widehat{f}_{\bsk} \wal_{\bsk}(\bsx)\textnormal{ absolutely convergent}\Big\}.
\end{align*}
and $\{e_{\bsk}\}_{\bsk\in\NN_0^s}$  with
$$
e_{\bsk}(\bsx)=\prod_{j=1}^se_{k_j}(x_j)= \wal_{\bsk}(\bsx):=
\prod_{j=1}^s \wal_{k_j}(x_j)
\ \ \ \ \mbox{for all}\ \
\bsx\in[0,1]^s
$$
as its orthonormal basis. Then $f\in H_s$ iff
$$
f(\bsx)=\sum_{\bsh\in\NN_0^s}
\widehat{f}_{\bsh,\wal}\,\wal_{\bsh}(\bsx)\ \ \ \ \mbox{for all}\ \
\bsx\in[0,1]^s
$$
with $\sum_{\bsh\in\NN_0^s}|\widehat{f}_{\bsh,\wal}|^2<\infty$. Here,
$$\widehat{f}_{\bsh,\wal}=\langle f,\wal_{\bsh} \rangle_{L_2}=
\int_{[0,1]^s}f(\bsx)\,\overline{\wal_{\bsh}(\bsx)}\,{\rm d}\bsx
$$
is the $\bsh$th Walsh coefficient. 

The weighted Walsh space $H_{s,\bsa,\bsb}$ is obtained by taking
$m_k\equiv 1$. Then $r_k=k$ and $k(n)=n$. The inner product of
$H_{s,\bsa,\bsb}$ for $f,g\in H_{s,\bsa,\bsb}$ is given by
$$
\left<f,g\right>_{H_{s,\bsa,\bsb}}=\sum_{\bsh\in\NN_0^s}
\omega^{-\sum_{j=1}^sa_j|h_j|^{b_j}}\, \widehat{f}_{\bsh,\wal}
\,\overline{\widehat{g}_{\bsh,\wal}}.
$$
The space $H_{s,\bsa,\bsb}$ is a reproducing kernel Hilbert space
and its reproducing kernel is
$$
K_{s,\bsa,\bsb}(\bsx,\bsy)=\sum_{\bsh\in\NN_0^s}
\omega^{\sum_{j=1}^sa_j|h_j|^{b_j}}\, 
\wal_{\bsh}(\bsx)\,\overline{\wal_{\bsh}(\bsy)} \ \
\ \ \mbox{for all}\ \ \bsx,\bsy\in[0,1]^s.
$$
More information 
on the 
Walsh spaces with finite smoothness can be found
in \cite{DP05,DP10}.\qed

\section{Multivariate Approximation}\label{secapprox}

By multivariate approximation we mean an embedding operator
$\APP_s:H_{s,\bsa,\bsb}\to H_s$ given by
$$
\APP_sf=f\ \ \ \ \mbox{for all} \ \ f\in H_{s,\bsa,\bsb}.
$$
Due to \eqref{6}, the operator $\APP_s$ is well defined, and it is
a continuous linear operator. Furthermore,
$\|\APP_sf\|_{H_s}\le\|f\|_{H_{s,\bsa,\bsb}}$ for all $f\in
  H_{s,\bsa,\bsb}$ and
$$
\|\APP_s\|=1\ \ \ \ \mbox{for all} \ \ s\in\NN.
$$
We will later show that $\APP_s$ is a compact operator.

We want to approximate $\APP_sf$ by algorithms
$A_n:H_{s,\bsa,\bsb}\to H_s$ that use at most $n$ continuous
linear functionals of $f$. Without loss of generality, see e.g.
\cite{NW08,TWW88}, we may restrict ourselves to linear algorithms
of the form
$$
A_nf=\sum_{j=1}^n L_j(f)\,g_j\ \ \ \ \mbox{for all}\ \ f\in
H_{s,\bsa,\bsb}
$$
for some $L_j\in H_{s,\bsa,\bsb}^*$ and $g_j\in H_s$ for
$j=1,2,\dots,n$.

We consider the worst case setting in which the error of $A_n$ is
defined as
$$
e(A_n)=\sup_{\|f\|_{H_{s,\bsa,\bsb}}\le
1}\,\|\APP_sf-A_nf\|_{H_s}= \|\APP_s-A_n\|.
$$
For $n=0$, we have the so-called initial error which is achieved
by the zero algorithm $A_0=0$, and $e(A_0)=\|\APP_s\|=1$.

By the $n$th minimal (worst case) error we mean the minimal error
among all algorithms~$A_n$,
$$
e(n,\APP_s)=\inf_{A_n}\,e(A_n).
$$
Clearly, $e(0,\APP_s)=1$. In a moment 
an algorithm
$A_n^*$ for which the infimum is attained will be presented.

By the information complexity $n(\eps,\APP_s)$ we mean the minimal
$n$ for which we can find an algorithm $A_n$ with error at most
$\eps\in(0,\infty)$,
$$
n(\eps,\APP_s)=\min\{\,n\, :\ e(n,\APP_s)\le \eps\,\}.
$$
Clearly, $n(\eps,\APP_s)=0$ for all $\eps\ge1$, and therefore the
only $\eps$'s of interest are from $(0,1)$.

It is well known, see again e.g., \cite{NW08,TWW88}, that the
$n$th minimal errors $e(n,\APP_s)$ and the information complexity
$n(\eps,\APP_s)$ depend on the eigenvalues of the continuous and
linear operator $W_s=\APP_s^*\APP_s:H_{s,\bsa,\bsb}\to
H_{s,\bsa,\bsb}$. The operator $W_s$ is self-adjoint and in a
moment we shall see that $W_s$ is also compact. Let
$(\lambda_{s,j},\eta_{s,j})$ be the eigenpairs of $W_s$,
$$
W_s\eta_{s,j}=\lambda_{s,j}\,\eta_{s,j}\ \ \ \ \mbox{for all}\  \
j\in\NN,
$$
where the eigenvalues $\lambda_{s,j}$ are ordered,
$$
\lambda_{s,1}\ge \lambda_{s,2}\ge\cdots\ge 0,
$$
and the eigenelements $\eta_{s,j}$ are orthonormal,
$$
\left<\eta_{s,j_1},\eta_{s,j_2}\right>_{H_{s,\bsa,\bsb}}=\delta_{j_1,j_2}
\ \ \ \ \mbox{for all}\ \ j_1,j_2\in\NN.
$$
Then the $n$th minimal error is attained for the algorithm
$$
A^*_nf=\sum_{j=1}^n\left<f,\eta_{s,j}\right>_{H_{s,\bsa,\bsb}}\,\eta_{s,j}
\ \ \ \ \mbox{for all}\ \ f\in H_{s,\bsa,\bsb},
$$
and
$$
e(n,\APP_s)=e(A_n^*)=\sqrt{\lambda_{s,n+1}}\ \ \ \ \mbox{for all}
\ \ n\in \NN_0.
$$
This implies that the information complexity is equal to
\begin{equation}\label{9}
n(\eps,\APP_s)=\min\{\,n\in\NN_0\,:\ \lambda_{s,n+1}\le
\eps^2\,\}.
\end{equation}

We now find the eigenpairs of $W_s$. Using the notation and
results of the previous section, we know that
$\{e_{\bsn,\bsa,\bsb}\}_{\bsn\in\NN_0^s}$ is an orthonormal basis
of $H_{s,\bsa,\bsb}$. We prove that
$$
W_se_{\bsn,\bsa,\bsb}=\omega^{\sum_{j=1}^sa_j[k(n_j)]^{b_j}}\,e_{\bsn,\bsa,\bsb}
\ \ \ \ \mbox{for all}\ \ \bsn\in\NN_0^s.
$$
Indeed, for $f,g\in H_{s,\bsa,\bsb}$ we have
$$
\left<\APP_sf,\APP_sg\right>_{H_s}=\left<f,
  \APP_s^*\APP_sg\right>_{H_{s,\bsa,\bsb}}=
\left<f,W_sg\right>_{H_{s,\bsa,\bsb}}.
$$
Taking $f=e_{\bsn_1,\bsa,\bsb}$ and $g=e_{\bsn_2,\bsa,\bsb}$ for
arbitrary $\bsn_1,\bsn_2\in\NN_0^s$ we obtain from \eqref{5},
\begin{eqnarray*}
\left<e_{\bsn_1,\bsa,\bsb},W_se_{\bsn_2,\bsa,\bsb}\right>_{H_{s,\bsa,\bsb}}
&=&\left<e_{\bsn_1,\bsa,\bsb},e_{\bsn_2,\bsa,\bsb}\right>_{H_s} \\
&=&\left(\prod_{j=1}^s\omega^{a_j[k((n_1)_j)]^{b_j}/2}\,
\omega^{a_j[k((n_2)_j)]^{b_j}/2}\right)\,\left<e_{\bsn_1},e_{\bsn_2}
\right>_{H_s}\\
&=&
\omega^{\sum_{j=1}^sa_j[k((n_1)_j)]^{b_j}/2+a_j[k((n_2)_j)]^{b_j}/2}\,
\delta_{\bsn_1,\bsn_2}.
\end{eqnarray*}
Hence,
$$
\left<e_{\bsn_1,\bsa,\bsb},W_se_{\bsn_2,\bsa,\bsb}\right>_{H_{s,\bsa,\bsb}}=0
\ \ \ \ \mbox{for all} \ \ \bsn_1\not=\bsn_2,
$$
and
$$
\left<e_{\bsn,\bsa,\bsb},W_se_{\bsn,\bsa,\bsb}\right>_{H_{s,\bsa,\bsb}}=
\omega^{\sum_{j=1}^sa_j[k(n_j)]^{b_j}}.
$$
This means that
$$
W_se_{\bsn,\bsa,\bsb}=\sum_{\bsn_1\in\NN_0^s}\left<W_se_{\bsn,\bsa,\bsb},
e_{\bsn_1,\bsa,\bsb}\right>_{H_{s,\bsa,\bsb}}\,e_{\bsn_1,\bsa,\bsb}=
\omega^{\sum_{j=1}^sa_j[k(n_j)]^{b_j}}\, e_{\bsn,\bsa,\bsb},
$$
as claimed. Hence,
$$
\left(\omega^{\sum_{j=1}^sa_j[k(n_j)]^{b_j}}, e_{\bsn,\bsa,\bsb}
\right)_{\bsn\in\NN_0^s}
$$
are the eigenpairs of $W_s$.

As an example consider the weighted Hermite space or the weighted
cosine space for which $m_k\equiv 1$. Then $k(n_j)=n_j$ and the
eigenpairs are of the form
$$
\left(\omega^{\sum_{j=1}^sa_jn_j^{b_j}},
\omega^{\sum_{j=1}^sa_jn_j^{b_j}/2}\,e_{\bsn}\right)_{\bsn\in\NN_0^s}.
$$
For the weighted Korobov space, we have $m_0=1$ and $m_k=2$ for
all $k\in\NN$. Then $k(n_j)=\lceil n_j/2\rceil$ and the eigenpairs
are of the form
$$
\left(\omega^{\sum_{j=1}^sa_j\lceil n_j/2\rceil^{b_j}},
\omega^{\sum_{j=1}^sa_j\lceil n_j/2\rceil^{b_j}/2}\,
e_{\bsn}\right)_{\bsn\in\NN_0^s}.
$$

We turn to the general case. The eigenvalues of $W_s$ may be
multiple. Indeed, for $n_j\in\NN_0$ we obtain the same $k(n_j)$
for all $n_j\in
\{r_{k(n_j)},r_{k(n_j)}+1,\ldots,r_{k(n_j)+1}-1\}$, i.e., for
$r_{k(n_j)+1}-r_{k(n_j)}=m_{k(n_j)}$ different values of $n_j$.
This means that $W_s$ has the eigenvalues
$$
\omega^{\sum_{j=1}^sa_jk_j^{b_j}} \ \ \mbox{of multiplicity}\ \
m_{\bsk}:=\sum_{\bsl\in\cB_{\bsk}}m_{l_1}m_{l_2}\cdots m_{l_s},
$$
where $\cB_{\bsk}=\{\bsl\in\NN_0^s\,:\,\sum_{j=1}^{s}a_jl_j^{b_j}=\sum_{j=1}^{s}a_jk_j^{b_j}\}$.
In particular, the largest eigenvalue $\lambda_{s,1}=1$, obtained
for $k_j=0$ for all $j=1,2,\dots,s$, has multiplicity $m_0^s$.
So for $m_0=1$ the largest eigenvalue is single.

Clearly, the sequence of ordered eigenvalues
$\{\lambda_{s,j}\}_{j\in\NN}$ is the same as the sequence
$\{\omega^{\sum_{j=1}^sa_j[k(n_j)]^{b_j}}\}_{\bsn\in\NN_0^s}$.
Furthermore it is obvious that $\lim_{j\to\infty}\lambda_{s,j}=0$,
which implies that $\APP_s$ as well as $W_s$ are compact.

We now find a more convenient formula for the information
complexity $n(\eps,\APP_s)$. {}From \eqref{9} we conclude that for
$\eps\in(0,\infty)$ we have
$$
n(\eps,\APP_s)=|\{\,j\in\NN_0\,:\ \lambda_{s,j}>\eps^2\,\}|,
$$
or equivalently
$$
n(\eps,\APP_s)=|\{\,j\in\NN_0\,:\
\log\,\lambda_{s,j}^{-1}<\log\,\eps^{-2}\,\}|.
$$
All eigenvalues $\lambda_{s,j}$ are of the form
$\omega^{\sum_{j=1}^sa_jk_j^{b_j}}$ with multiplicity $m_{\bsk}$
for $\bsk\in\NN_0^s$. Therefore
$$
\log\,\omega^{-\sum_{j=1}^sa_jk_j^{b_j}}
=\left(\sum_{j=1}^sa_jk_j^{b_j}\right)\, \log\,\omega^{-1}
$$
and
$$
\log\,\omega^{-\sum_{j=1}^sa_jk_j^{b_j}}<\log\,\eps^{-2}\ \ \
\mbox{iff}\ \ \ \sum_{j=1}^sa_jk_j^{b_j} <
\frac{\log\,\eps^{-2}}{\log\,\omega^{-1}}.
$$
Let
$$
A(\eps,s)=\left\{\,\bsk\in\NN_0^s\,:\ \sum_{j=1}^sa_jk_j^{b_j} <
\frac{\log\,\eps^{-2}}{\log\,\omega^{-1}}\,\right\}.
$$
Then
\begin{equation}\label{10}
n(\eps,\APP_s)=\sum_{\bsk\in A(\eps,s)}m_{k_1}m_{k_2}\cdots
m_{k_s}.
\end{equation}
Note that for $m_k\equiv 1$, as e.g. for the weighted Hermite space and
the weighted cosine space, we have
$$
n(\eps,\APP_s)=|A(\eps,s)|.
$$

For the general case, the set $A(\eps,s)$ is empty for $\eps\ge1$,
and then $n(\eps,\APP_s)=0$ as we already remarked. Let
\begin{equation}\label{functionx}
x(t)=\frac{\log\,t^{-2}}{\log\,\omega^{-1}}\ \ \ \ \mbox{for all}\
\ t\in(0,\infty).
\end{equation}
For $s=1$, it is easy to check that $A(\eps,1)=\{0,1,\dots,\lceil
(x(\eps)/a_1)^{1/b_1}\rceil -1\}$ and
\begin{equation}\label{s=1}
n(\eps,\APP_1)=m_0+m_1+\dots+m_{\lceil
(x(\eps)/a_1)^{1/b_1}\rceil-1}.
\end{equation}
For $s\ge2$, we have
\begin{eqnarray*}
A(\eps,s)&=&\bigcup_{k=0}^\infty\left\{\,\bsk\in\NN_0^{s-1}\times\{k\}\,:\
\sum_{j=1}^{s-1}a_jk_j^{b_j}<x(\eps)-a_sk^{b_s}\,\right\}\\
&=& \bigcup_{k=0}^{\lceil (x(\eps)/a_s)^{1/b_s}\rceil -1}
\left\{\,\bsk\in \NN_0^{s-1}\times\{k\}\,:\
\sum_{j=1}^{s-1}a_jk_j^{b_j}<x(\eps)-a_sk^{b_s}\,\right\}.
\end{eqnarray*}
Since $x(\eps)-a_sk^{b_s}=(\log\,
(\eps\,\omega^{-a_sk^{b_s}/2})^{-2})/\log\,\omega^{-1}$, we
obtain from \eqref{10}
\begin{equation}\label{12}
n(\eps,\APP_s)=\sum_{k=0}^{\lceil(x(\eps)/a_s)^{1/b_s}\rceil -1}
m_k\,n\left( \eps\,\omega^{-a_s{k}^{b_s}/2} ,\APP_{s-1}\right).
\end{equation}
For $\eps_1\le\eps_2$ we have $n(\eps_2,\APP_s)\le
n(\eps_1,\APP_s)$. Since  $\eps\le \eps\,\omega^{-a_s{k}^{b_s}/2}$
for all $k\in\NN_0$, we conclude that
\begin{equation}\label{rec}
n(\eps,\APP_s)\le\left(\sum_{k=0}^{\lceil(x(\eps)/a_s)^{1/b_s}\rceil
    -1}\,m_k\right)\,n(\eps,\APP_{s-1})\ \ \ \ \mbox{for all}\ \ s\ge2.
\end{equation}
We obtain a lower bound on $n(\eps,\APP_s)$ if we consider only
the term $k=0$ in \eqref{12}. Then
$$
n(\eps,\APP_s)\ge m_0\,n(\eps,\APP_{s-1})\ \ \ \ \mbox{for all}\ \
s\ge2.
$$
For $\eps\ge\omega^{a_s/2}$ we have
$$
n(\eps,\APP_s)=m_0\,n(\eps,\APP_{s-1})\ \ \ \ \mbox{for all}\ \
s\ge2
$$
since $\eps\,\omega^{-a_s k^{b_s}/2}\ge1$ for all positive $k$ and
the terms in \eqref{12} for $k>0$ are zero.

For $x(\eps)>a_1$, define
\begin{equation}\label{13}
j(\eps)=\sup\{\,j\in \NN\,:\ x(\eps)>a_j\,\}.
\end{equation}
Obviously, $j(\eps)\ge1$. For $\lim_ja_j<\infty$ (here and in the following we use the convention that we write $\lim_j$ instead of $\lim_{j\to\infty}$), we have
$j(\eps)=\infty$ for small $\eps$. On the other hand, if
$\lim_ja_j=\infty$ we can replace the supremum in the definition
of $j(\eps)$ by the maximum and $j(\eps)$ is finite for all $\eps$
with $x(\eps)>a_1$. However, $j(\eps)$ tends to infinity as $\eps$
tends to zero.

If $j(\eps)$ is finite then
$$
n(\eps,\APP_s)=m_0^{s-j(\eps)}\,n(\eps,\APP_{j(\eps)})\ \ \ \
\mbox{for all}\ \ s\ge j(\eps).
$$
Indeed, for $j\in(j(\eps),s]$ we have $x(\eps)\le a_j$ and
$x(\eps)-a_jk^{b_j}\le0$ for all $k\ge1$. This implies that
$\eps\,\omega^{-a_jk^{b_j}/2}\ge1$ for all $k\ge1$, and the sum in
\eqref{12} reduces to one term for $k=0$. Hence,
$n(\eps,\APP_s)=m_0\,n(\eps,\APP_{s-1})=
\cdots=m_0^{s-j(\eps)}n(\eps,\APP_{j(\eps)})$, as claimed.
Therefore, if $j(\eps)<\infty$ and $m_0=1$ then $n(\eps,\APP_s)$
is independent of $s$ for large $s$, and
$$
\lim_{s\to\infty}\frac{\log\,n(\eps,\APP_s)}{s}=0.
$$
Recall that we assume that the sequence $\bsm=\{m_k\}_{k\in\NN_0}$
of multiplicities is bounded. That is,
$$
\mmax=\max_{k\in\NN}\ m_k
$$
is well defined and $\mmax<\infty$. We also set
$$
\mmin=\min_{k\in\NN}\ m_k.
$$
Clearly, $\mmin\ge1$.

We are ready to prove the following lemma.
\begin{lem}\label{lem1}

\qquad

Let $x(\eps)$, $j(\eps)$, $\mmax$ and $\mmin$ be defined as above.
\begin{enumerate}[(i)]
\item\label{lem1p1} For $\eps\in(0,1)$ we have
$$
n(\eps,\APP_s)\ge m_0^s,
$$
whereas for $\eps\in(0,1)$ and $x(\eps)\le a_1$ we have
$$
n(\eps,\APP_s)= m_0^s.
$$
\item\label{lem1p2} For $x(\eps)>a_1+a_2+\dots+ a_s$ we have
$$
n(\eps,\APP_s)\ge (m_0+m_1)^s.
$$
\item\label{lem1p3} For $x(\eps)>a_1$ and $\eps\in(0,1)$ we have
$$
n(\eps,\APP_s)\le m_0^s\,\prod_{j=1}^{\min(s,j(\eps))}\left(1+
\frac{\mmax}{m_0}\,\left(\left\lceil
\left(\frac{x(\eps)}{a_j}\right)^{1/b_j}\right\rceil-1\right)\right).
$$
\item\label{lem1p4} For $x(\eps)>a_1$, $\eps\in(0,1)$, and arbitrary
  $\alpha_j\in[0,1]$ we have
$$
n(\eps,\APP_s)\ge m_0^s\,\prod_{j=1}^{\min(s,j(\eps))}\left(1+
\frac{\mmin}{m_0}\,\left(\left\lceil
\left(\frac{x(\eps)}{a_j}(1-\alpha_j)\prod_{k=j+1}^s\alpha_k
\right)^{1/b_j}\right\rceil-1\right)\right).
$$
In particular, for $\alpha_j=(j-1)/j$ we have
$$
n(\eps,\APP_s)\ge m_0^s\,\prod_{j=1}^{\min(s,j(\eps))}\left(1+
\frac{\mmin}{m_0}\,\left(\left\lceil
\left(\frac{x(\eps)}{a_j\,s}\right)^{1/b_j}\right\rceil-1\right)\right).
$$
\end{enumerate}
\end{lem}
\begin{proof}
To prove \eqref{lem1p1}, observe that for $\eps\in(0,1)$
 the set $A(\eps,s)$ 
is nonempty since
$\bsk={\bf 0}\in A(\eps,s)$. Therefore \eqref{10} yields
$n(\eps,\APP_s)\ge m_0^s$. Furthermore, for $x(\eps)\le a_1$ the
set $A(\eps,s)=\{{\bf 0}\}$ and therefore $n(\eps,\APP_s)=m_0^s$, as
claimed.

To prove \eqref{lem1p2}, observe that all $\bsk\in\{0,1\}^s$
belong to the set $A(\eps,s)$. Therefore
$$
n(\eps,\APP_s)\ge \sum_{k_1,k_2,\dots,k_s=0}^1m_{k_1}m_{k_2}\cdots
m_{k_s}=(m_0+m_1)^s,
$$
as claimed.

To prove \eqref{lem1p3}, we first take $s=1$. Then \eqref{s=1}
yields
$$
n(\eps,\APP_1)\le m_0+\mmax\,\left(\left\lceil
    \left(\frac{x(\eps)}{a_1}\right)^{1/b_1}\right\rceil -1\right)=
m_0\left(1+\frac{\mmax}{m_0}\,\left(\left\lceil
    \left(\frac{x(\eps)}{a_1}\right)^{1/b_1}\right\rceil -1\right)\right),
$$
as needed. For $s\ge2$ we use \eqref{rec} and obtain
$$
n(\eps,\APP_s)\le m_0\left(1+\frac{\mmax}{m_0}\, \left(\left\lceil
\left(\frac{x(\eps)}{a_s}\right)^{1/b_s}\right\rceil-1\right)\right)\,
n(\eps,\APP_{s-1}).
$$
This implies that
$$
n(\eps,\APP_s)\le m_0^s\,\prod_{j=1}^{s}\left(1+
\frac{\mmax}{m_0}\,\left(\left\lceil
\left(\frac{x(\eps)}{a_j}\right)^{1/b_j}\right\rceil-1\right)\right).
$$
Note that for $j(\eps)< s$ we have $x(\eps)\le a_j$ for all $j\in
[j(\eps)+1,s]$ and therefore
$$
1+\frac{\mmax}{m_0}\,\left(\left\lceil
\left(\frac{x(\eps)}{a_j}\right)^{1/b_j}\right\rceil-1\right)=1.
$$
This means that we can restrict the product to $j$ up to
$j(\eps)$. This completes the proof of \eqref{lem1p3}.

To prove \eqref{lem1p4}, it is enough to prove that
$$
n(\eps,\APP_s)\ge m_0^s\,\prod_{j=1}^{s}\left(1+
\frac{\mmin}{m_0}\,\left(\left\lceil
\left(\frac{x(\eps)}{a_j}(1-\alpha_j)\prod_{k=j+1}^s\alpha_k
\right)^{1/b_j}\right\rceil-1\right)\right)
$$
since for $j\in(j(\eps),s]$ we have $x(\eps)\le a_j$ and the
corresponding factors are one.

Take first $s=1$. Then \eqref{s=1} yields
\begin{eqnarray*}
n(\eps,\APP_1)&\ge& m_0\left(1+\frac{\mmin}{m_0}\left(\left\lceil
\left(\frac{x(\eps)}{a_1}\right)^{1/b_1}\right\rceil -1 \right)\right)\\
&\ge& m_0\left(1+\frac{\mmin}{m_0}\left(\left\lceil
\left(\frac{x(\eps)}{a_1}(1-\alpha_1) \right)^{1/b_1}\right\rceil
-1 \right)\right),
\end{eqnarray*}
as needed. For $s\ge2$, note that $x(\eps)-a_sk^{b_s}>\alpha_s
x(\eps)$ for all $k\in\NN$ for which
$$
k\le \left\lceil
  \left(\frac{x(\eps)(1-\alpha_s)}{a_s}\right)^{1/b_s}\right\rceil -1.
$$
For such $k$ we have $\eps\omega^{-a_sk^{b_s}/2}<\eps^{\alpha_s}$
and therefore from \eqref{12} we obtain
\begin{eqnarray*}
n(\eps,\APP_s)&\ge&
\sum_{k=0}^{\lceil(x(\eps)(1-\alpha_s)/a_s)^{1/b_s}\rceil -1}m_k\,
n(\eps^{\alpha_s},\APP_{s-1})\\
&\ge&m_0\left(1+\frac{\mmin}{m_0}\,\left( \left\lceil
  \left(\frac{x(\eps)(1-\alpha_s)}{a_s}\right)^{1/b_s}\right\rceil
-1\right)\right)\, n(\eps^{\alpha_s},\APP_{s-1}).
\end{eqnarray*}
Since $x(\eps^{\alpha_s})=\alpha_s x(\eps)$, the proof is
completed by applying induction on $s$. For $\alpha_j=(j-1)/j$ we
have $(1-\alpha_j)\prod_{k=j+1}^s\alpha_k=1/s$, which completes
the proof.
\end{proof}

\section{Exponential Convergence}\label{secexp}
As in \cite{DKPW13,DLPW11,IKLP14,KPW14}, by exponential
convergence (EXP) we mean that the $n$th minimal errors
$e(n,\APP_s)$ are bounded by
$$
e(n,\APP_s)\le C_s\,q^{(n/M_s)^{p_s}}\ \ \ \ \mbox{for all}\ \
n\in\NN,
$$
for some positive $C_s,M_s$ and $p_s$ with $q\in(0,1)$. The
supremum of $p_s$ for which the last bound holds is denoted by
$p_s^*$ and is called the exponent of EXP for the $s$-variate
case. We also have the concept of uniform exponential convergence
(UEXP)  if we can take $p_s=p>0$ for all $s\in\NN$. Then the
supremum of such $p$ is denoted by $p^*$ and is called the
exponent of UEXP.

We want to verify when EXP and UEXP hold for the approximation
problem $\APP=\{\APP_s\}_{s\in\NN}$ in terms of the varying
parameters $\bsa=\{a_s\}_{s\in\NN}$ and $\bsb=\{b_s\}_{s\in\NN}$,
which define the domain spaces $H_{s,\bsa,\bsb}$ of $\APP_s$ and
satisfy \eqref{3}.

\begin{thm}\label{thm1}

\qquad

Consider the approximation problem $\APP=\{\APP_s\}_{s\in \NN}$
with the embedding operators $\APP_s:H_{s,\bsa,\bsb}\to H_s$. Then
\begin{enumerate}[(i)]
\item\label{th1exp} EXP holds for arbitrary $\bsa$ and $\bsb$ with the exponent
$$
p^*_s=\frac1{B_s}\ \ \ \ \mbox{and}\ \ \ \
B_s:=\sum_{j=1}^s\frac1{b_j}.
$$
\item\label{th1uexp} UEXP holds iff $\bsa$ is arbitrary and $\bsb$ is such that
$$
B:=\sum_{j=1}^\infty\frac1{b_j}<\infty.
$$
If $B<\infty$ then the exponent of UEXP is $p^*=1/B$.
\end{enumerate}
\end{thm}
\begin{proof}
From \eqref{lem1p3} and \eqref{lem1p4} of Lemma~\ref{lem1} with a
fixed $s$ we conclude that there are positive numbers $c_1(s)$ and
$c_2(s)$ such that
$$
c_1(s)[x(\eps)]^{B_s}\le n(\eps,\APP_s)\le c_2(s)\,[x(\eps)]^{B_s}
\ \ \ \ \mbox{for all}\ \ \eps\in(0,1).
$$
Clearly, $x(\eps)=\Theta(\log\,\eps^{-1})$. Therefore
$$
n(\eps,\APP_s)=\Theta\left([\log\,\e^{-1}]^{B_s}\right).
$$
From this it follows that we can
find positive $c_j(s)$ for $j=3,4,5,6$ such that
$$
c_3(s)\,{\rm e}^{-(n/c_4(s))^{1/B_s}}\le e(n,\APP_s)\le
c_5(s)\,{\rm e}^{-(n/c_6(s))^{1/B_s}}\ \ \ \ \mbox{for all}\ \
n\in\NN,
$$
where ${\rm e}=\exp(1)$. This proves EXP with $p^*_s=1/B_s$, as
claimed.

We now turn to UEXP. Suppose that UEXP holds. Then $e(n,\APP_s)\le
C_s\,q^{(n/M_s)^p}$. This implies that
$$
n(\eps,\APP_s)=\Theta \left([\log\,\e^{-1}]^{1/p}\right).
$$
Thus, $B_s\le 1/p$  for all $s\in\NN$. Therefore $B\le 1/p<\infty$
and $p^*\le 1/B$. On the other hand, if $B<\infty$ then we can set
$p_s=1/B$ and obtain UEXP. Hence, $p^*\ge 1/B$, and therefore
$p^*=1/B$. This completes the proof.
\end{proof}
We stress that EXP and UEXP hold for arbitrary sequences $\bsm$ of
multiplicity and the only condition is on $\bsb$ for UEXP. This is
true since the concepts of EXP and UEXP do not specify how $C_s$
and $M_s$ depend on $s$. In fact, in general, it is easy to see
from Lemma~\ref{lem1} that $c_1(s)$ and $c_2(s)$, as well as the other
$c_j(s)$,  depend exponentially on $s$. It is especially clear if
the multiplicity $m_0\ge2$. If we wish to control the dependence
on $s$ and to control
the exponential dependence on $s$ then we
need to study tractability which is the subject of the next
section.

\section{Tractability}\label{sectract}

Tractability studies how the information complexity depends on
both $\eps^{-1}$ and $s$. The key point is to characterize when
this dependence is not exponential in $(s^{\,t_1},\eps^{-t_2})$ or in
$(s^{\,t_1},(1+\log\,\eps^{-1})^{t_2})$ for some positive $t_1$ and
$t_2$,
and when this dependence is polynomial in
$(s,\eps^{-1})$ or in $(s,1+\log\,\eps^{-1})$. 
For $t_1=t_2=1$, the survey of
tractability results for general multivariate problems and for the
pair $(s,\eps^{-1})$ can be found in \cite{NW08,NW10,NW12}, and
for more specific multivariate problems and the pair
$(s,1+\log\,\eps^{-1})$ in \cite{DKPW13,DLPW11,IKLP14,KPW14}.

We will cover a number of tractability notions and verify when
they hold for the approximation problem
$\APP=\{\APP_s\}_{s\in\NN}$ in terms of the parameters
$\bsa,\bsb$, $\bsm$ and $\omega$. We will analyze the tractability
notions starting from the weakest notions and continuing to the strongest
notions. A table which gives an overview of the obtained
tractability results is presented in Section~\ref{sec_sum}.

\subsection{Standard Notions of Tractability}
By the standard notions of tractability we mean tractability
notions with respect to the pair $(s,\eps^{-1})$.

\begin{itemize}
\item {\bf $(t_1,t_2)$-Weak Tractability}

As in \cite{SW14}, we say that $\APP$ is $(t_1,t_2)$-weakly
tractable (shortly $(t_1,t_2)$-WT) for positive $t_1$ and $t_2$
iff
$$
\lim_{s+\eps^{-1}\to\infty}\frac{\log\,n(\eps,\APP_s)}{s^{t_1}+\eps^{-t_2}}=0.
$$
This means that $n(\eps,\APP_s)$ is \emph{not} exponential in
$s^{\,t_1}$ and $\eps^{-t_2}$ but it may be exponential in
$s^{\tau_1}$ or $\eps^{-\tau_2}$ for positive $\tau_1<t_1$ or
$\tau_2<t_2$. In particular, if $t_1>1$ we may have the
exponential dependence on $s$ which is called the curse of
dimensionality.

\begin{thm}\label{thm2}

\qquad

\qquad $\APP$ is {\rm $(t_1,t_2)$-WT} for the parameters
$\bsa,\bsb,\bsm$ and $\omega$ \ \ iff \ \ $t_1>1$ or $m_0=1$.
\end{thm}

\begin{proof}
Suppose that $\APP$ is $(t_1,t_2)$-WT for the parameters
$\bsa,\bsb,\bsm$ and $\omega$.
Then for fixed $\varepsilon\in (0,1)$ 
we obtain from \eqref{lem1p1} of Lemma~\ref{lem1} that 
$$0=\lim_{s \rightarrow \infty} 
\frac{n(\varepsilon,\APP_s)}{s^{t_1} + 
\varepsilon^{-t_2}} \ge \lim_{s \rightarrow \infty} 
\frac{s \log m_0}{s^{t_1} + \varepsilon^{-t_2}}=\lim_{s \rightarrow
  \infty} s^{1-t_1} \log m_0$$ 
and hence we must have $t_1>1$ or $m_0=1$. 

Suppose now that $t_1>1$. We first show that the hardest case of
$\APP$ is for constant $\bsa$ and $\bsb$, i.e., $a_j\equiv a_1$
and $b_j\equiv b_1$. Indeed, the eigenvalues of $W_s$, which
define $n(\eps,\APP_s)$, are
$\omega^{\sum_{j=1}^sa_j[k(n_j)]^{b_j}}$. Clearly,
$$
\sum_{j=1}^sa_j[k(n_j)]^{b_j}\ge \sum_{j=1}^sa_1[k(n_j)]^{b_0},
$$
where $b_0=\inf_jb_j$. Due to \eqref{3}, we have $b_0>0$.
Therefore
$$
\omega^{\sum_{j=1}^sa_j[k(n_j)]^{b_j}}\le
\omega^{\sum_{j=1}^sa_1[k(n_j)]^{b_0}},
$$
and $n(\eps,\APP_s)$ is maximized for $a_j\equiv a_1$ and
$b_j\equiv b_0$ (and just now $b_0=b_1$).

Hence, it is enough to show $(t_1,t_2)$-WT for constant $\bsa$ and
$\bsb$. From \eqref{lem1p3} of Lemma~\ref{lem1} we have
$$
\log\,n(\eps,\APP_s)\le s\,\left( \log\,m_0\ +\
\log\left(1+\frac{\mmax2^{1/b_1}}{m_0(a_1\log\,\omega^{-1})^{1/b_1}}\,
[\log\,\eps^{-1}]^{1/b_1}\right)\right).
$$
This shows that for small $\eps$ we have
\begin{equation}\label{generalbound}
\log\,n(\eps,\APP_s)=\mathcal{O}(s\,\log\,\log\,\eps^{-1})
\end{equation}
with the factor in the big $\mathcal{O}$ notation independent of
$s$ and $\eps^{-1}$. Hence,
$$
\frac{\log\,n(\eps,\APP_s)}{s^{t_1}+\eps^{-t_2}}=
\mathcal{O}\left(\frac{s\,\log\,\log\,\eps^{-1}}{s^{t_1}+\eps^{-t_2}}\right).
$$
Let $y=\max(s^{t_1},\eps^{-t_2})$. 
Then $\eps^{-1}\le y^{1/t_2}$,
$s\le y^{1/t_1}$ and 
$$
\frac{s\,\log\,\log\,\eps^{-1}}{s^{t_1}+\eps^{-t_2}}\le
\frac{y^{1/t_1}\,\log\,\log\,y^{1/t_2}}{y}=
\frac{\log\,\log\,y^{1/t_2}}{y^{1-1/t_1}}
$$
and it goes to zero as $s+\eps^{-1}$, or equivalently $y$,
approaches infinity since
$t_1>1$ and $t_2>0$. This proves $(t_1,t_2)$-WT.

Finally, suppose that $m_0=1$.
Then the second largest eigenvalue for all $s$ is
$\lambda_{1,2}=\omega^{a_1}$, which is smaller than the largest 
eigenvalue $\lambda_{s,1}=1$. 
As above it suffices to consider $\APP$ for constant $\bsa$ and
$\bsb$, i.e. $a_j \equiv a_1$ and $b_j \equiv b_1$. 
In this case, we can use 
an estimate for the information complexity 
which has been shown in \cite[p.~611]{PP14},
and which states 
$$
n(\varepsilon,\APP_s) \le \frac{s!}{(s-a_s(\varepsilon))!} 
\prod_{j=1}^{a_s(\varepsilon)} n(\varepsilon^{1/j},\APP_1),
$$ 
where 
$$
a_s(\varepsilon)= \min\left\{s,\left\lfloor 2\, 
\frac{\log \varepsilon^{-1}}{\log \omega^{-a_1}}
\right\rfloor -1\right\}.
$$ 
Then we have
\begin{equation}\label{prthm2a}
\log n(\varepsilon,\APP_s) 
\le \log \frac{s!}{(s-a_s(\varepsilon))!} +  
\sum_{j=1}^{a_s(\varepsilon)} \log n(\varepsilon^{1/j},\APP_1).
\end{equation}
{}From \eqref{s=1} 
with the assumption $m_0=1$ we obtain 
$$
n(\varepsilon^{1/j},\APP_1) \le 1+
m_{\max}\left(\frac{x(\varepsilon^{1/j})}{a_1}\right)^{1/b_1}\le 
2 \,m_{\max}\,\max\left(1,
\left(\frac{x(\varepsilon^{1/j})}{a_1}\right)^{1/b_1}\right).
$$ 
Note that
$$
\left(\frac{x(\varepsilon^{1/j})}{a_1}\right)^{1/b_1}=
\left(\frac2{a_1\,j}\,\frac{\log\,\varepsilon^{-1}}{\log\,\omega^{-1}}
\right)^{1/b_1}
\le
\left(\frac2{a_1}\,\frac{\log\,\varepsilon^{-1}}{\log\,\omega^{-1}}
\right)^{1/b_1}.
$$
Assume that $\varepsilon\le \omega^{a_1/2}$. Then the last
right hand side is at least one and therefore
\begin{eqnarray*}
\log n(\varepsilon^{1/j},\APP_1) & \le & \log
\left(2\, 
m_{\max} \left(\frac{2}{a_1}\right)^{1/b_1}\right)+\frac{1}{b_1} 
\log \left(\frac{\log \varepsilon^{-1}}{\log \omega^{-1}}\right)\\
& = & C_1 +C_2 \log \log \varepsilon^{-1},
\end{eqnarray*}
where $C_1=\log(2 \, m_{\max} (\tfrac{2}{a_1})^{1/b_1})-
\tfrac{1}{b_1}\log \log \omega^{-1}$ and $C_2=\tfrac{1}{b_1}$. 
Hence we obtain
\begin{eqnarray}\label{prthm2b}
\sum_{j=1}^{a_s(\varepsilon)} \log n(\varepsilon^{1/j},\APP_1) & \le & 
a_s(\varepsilon) (C_1+C_2 \log \log \varepsilon^{-1})\nonumber\\
& \le & C_3 \log \varepsilon^{-1} \log \log \varepsilon^{-1} 
\end{eqnarray}
with a suitable $C_3>0$. Hence we have
$$
\limsup_{s+\varepsilon^{-1}\rightarrow \infty} 
\frac{\log n(\varepsilon,\APP_s)}{s^{t_1}+\varepsilon^{-t_2}} \le
\limsup_{s+\varepsilon^{-1}\rightarrow \infty} 
\frac{\log \frac{s!}{(s-a_s(\varepsilon))!}}{s^{t_1}+
  \varepsilon^{-t_2}}.
$$
Since 
$$
\frac{s!}{(s-a_s(\varepsilon))!}=(s-a_s(\varepsilon)+1)(s-a_s(\varepsilon)+2)
\cdots s\le
s^{a_s(\varepsilon)}
$$
we have  
\begin{equation}\label{prthm2c}
\log\,\frac{s!}{(s-a_s(\varepsilon))!}\le a_s(\varepsilon)\,\log\,s=
\cO(\log\,\varepsilon^{-1}\,\log\,s).
\end{equation}
As before, let $y=\max(s^{t_1},\eps^{-t_2})$. 
Since $t_1>0$ and $t_2>0$ we have $\eps^{-1}\le y^{1/t_2}$,
$s\le y^{1/t_1}$ and 
$$
\frac{\log\,\eps^{-1}\ \log\,s}{s^{t_1}+\eps^{-t_2}}\le
\frac{[\log\,y]^2}{t_1t_2\,y}
$$
goes to zero as $s^{t_1}+\eps^{-t_2}$, or equivalently $y$, approaches
  infinity.
Hence 
$$
\limsup_{s+\varepsilon^{-1}\rightarrow \infty} \frac{\log
  n(\varepsilon,\APP_s)}{s^{t_1}+\varepsilon^{-t_2}} =
\lim_{s+\varepsilon^{-1}\rightarrow \infty} \frac{\log
  n(\varepsilon,\APP_s)}{s^{t_1}+\varepsilon^{-t_2}} 
=0.
$$
\end{proof}

\item{\bf Weak and Uniform Weak Tractability}

Weak tractability (WT) corresponds to $(t_1,t_2)$-WT for
$t_1=t_2=1$. Uniform weak tractability (UWT) holds iff
we have $(t_1,t_2)$-WT for all $t_1,t_2\in(0,1]$. 

\begin{thm}\label{thm3}

\qquad

\qquad $\APP$ is {\rm WT}\quad iff\quad $\APP$ is {\rm UWT} \quad iff\quad $m_0=1$.
\end{thm}
\begin{proof}
Since UWT implies WT, it is enough to show that WT implies 
$m_0=1$, and that $m_0=1$ implies UWT.
Suppose then 
that $\APP$ is WT. From the previous proof we conclude
that $m_0=1$. On the other hand, if $m_0=1$ then $\APP$ is not
only UWT but it is quasi-polynomially tractable which is a stronger
notion than UWT. This will be shown in a moment.
\end{proof}
\item {\bf Quasi-Polynomial Tractability}

$\APP$ is quasi-polynomially tractable (QPT) iff there are
positive numbers $C$ and $t$ such that
$$
n(\eps,\APP_s)\le
C\,\exp\left(t(1+\log\,s)(1+\log\,\eps^{-1})\right) \ \ \ \
\mbox{for all}\ \ s\in\NN,\ \eps\in(0,1).
$$
The infimum of $t$ satisfying the bound above is denoted by $t^*$,
and is called the exponent of QPT. Clearly, QPT implies UWT.

\begin{thm}\label{thm4}

\qquad

\qquad $\APP$ is {\rm QPT} \ \ iff\ \ $m_0=1$.

\qquad If $m_0=1$ then the exponent of {\rm QPT} is
 $t^*\le\frac2{a_1\,\log\,\omega^{-1}}$ and the last bound becomes an
 equality for constant $\bsa$ and $\bsb$.
\end{thm}
\begin{proof}
Suppose that $\APP$ is QPT. Then $\APP$ is UWT and $m_0=1$.

We now show that $m_0=1$ implies QPT and $t^*\le
2/(a_1\,\log\,\omega^{-1})$. As before, it is enough to prove it
for constant $\bsa$ and $\bsb$. In this case $H_{s,\bsa,\bsb}$ is
the tensor product of $s$ copies of $H_{a_1,b_1}$. Then the
eigenvalues $\{\lambda_{s,k}\}_{k\in \NN}$ of $W_s$ are products
of the eigenvalues $\{\lambda_k\}_{k\in \NN}$ of $W_1$, i.e.,
$\{\lambda_{s,k}\}_{k\in\NN}=\{\lambda_{k_1}\lambda_{k_2}
\cdots\lambda_{k_s}\}_{k_1,k_2,\dots,k_s\in\NN}$ with the ordered (distinct)
eigenvalues $\lambda_k=\omega^{a_1(k-1)^{b_1}}$ for $k\in \NN$. It
is proved in \cite{GW11} that $\APP$ is QPT iff
$\lambda_2<\lambda_1$  and $\mbox{decay}_\lambda:=\sup\{r\,:\
\lim_kk^r\lambda_k=0\}>0$. If so then
$$
t^*=\max\left(\frac2{\mbox{decay}_\lambda},\frac2{\log\,
\frac{\lambda_1}{\lambda_2}}\right).
$$
In our case, $\lambda_1=1$ and $\lambda_2=\omega^{a_1}$ so that
the assumption $\lambda_2<\lambda_1$ holds. Furthermore
$\lim_kk^r\,\omega^{a_1(k-1)^{b_1}}=0$ for all $r>0$, so that
$\mbox{decay}_\lambda=\infty$. Hence,
$t^*=2/(a_1\,\log\,\omega^{-1})$, as claimed.
\end{proof}

\item {\bf Polynomial and Strong Polynomial Tractability}

$\APP$ is polynomially tractable (PT) iff there are positive $C,p$
and $q\ge0$ such that
$$
n(\eps,\APP_s)\le C\,s^{\,q}\,\eps^{-p}\ \ \ \ \mbox{for all}\ \
s\in \NN,\ \eps\in(0,1).
$$
$\APP$ is strongly polynomially tractable (SPT) iff the last bound
holds for $q=0$. Then the infimum of $p$ in the bound above is
denoted by $p^*$, and is called the exponent of SPT. For
simplicity, we assume that
$$
\alpha:=\lim_{j\to\infty}\ \frac{a_j}{\log\,j}
$$
exists.

\begin{thm}\label{thm5}

\qquad

\qquad $\APP$ is {\rm SPT}\ \ iff\ \ $\APP$ is {\rm PT}\ \ iff
$$
m_0=1\ \ \ \mbox{and}\ \ \ \alpha>0.
$$
\qquad If this is the case then the exponent of {\rm SPT} is
$p^*=\frac2{\alpha\,\log\,\omega^{-1}}$.
\end{thm}
\begin{proof}
We use \cite[Theorem~5.2]{NW08} which states necessary and
sufficient conditions on PT and SPT in terms of the eigenvalues
$\{\lambda_{s,k}\}_{k\in \NN}$ of $W_s$. For our problem we have
$\lambda_{s,1}=1$. Namely, $\APP$ is PT iff there are numbers
$q\ge0$ and $\tau>0$ such that
\begin{equation}\label{ptspt}
\sup_{s\in\NN}
\,\left(\sum_{k=1}^\infty\lambda_{s,k}^\tau\right)^{1/\tau}\,s^{\,-q}
 <\infty,
\end{equation}
and it is SPT iff the last inequality holds with $q=0$. Then the
exponent $p^*$ of SPT is the infimum of $2\tau$ for $\tau$
satisfying \eqref{ptspt} with $q=0$.

In our case,
$$
\sum_{k=1}^\infty\lambda_{s,k}^\tau=\sum_{\bsn\in\NN_0^s}
\prod_{j=1}^s\omega^{\tau\,a_j[k(n_j)]^{b_j}}=
\prod_{j=1}^s\left(m_0+\sum_{k=1}^\infty
m_k\omega^{\tau\,a_j\,k^{b_j}}\right).
$$
Let $b_0=\inf_jb_j$ and
$C_{\tau}=\mmax\,\sum_{k=1}^\infty\omega^{\tau a_1(k^{b_0}-1)}$.
Then $b_0>0$ and $C_{\tau}<\infty$ for all $\tau>0$. Furthermore,
$$
m_1\,\omega^{\tau\,a_j}\le\sum_{k=1}^\infty
m_k\omega^{\tau\,a_j\,k^{b_j}}\le
\mmax\,\omega^{\tau\,a_j}\,\sum_{k=1}^\infty\omega^{\tau\,a_j(k^{b_j}-1)}
\le C_{\tau}\,\omega^{\tau\,a_j}.
$$
Therefore
\begin{equation}\label{ccc}
\prod_{j=1}^s\left(m_0+m_1\omega^{\tau\,a_j}\right)\le
\prod_{j=1}^s\left(m_0+\sum_{k=1}^\infty
m_k\omega^{\tau\,a_j\,k^{b_j}}\right) \le\prod_{j=1}^s\left(m_0+
C_{\tau}\,\omega^{\tau\,a_j}\right).
\end{equation}
Let $\omega^{\tau\,a_j}=(j+1)^{-x_j}$. That is,
$$
x_j=\frac{ a_j}{\log(j+1)}\, \tau\,\log\,\omega^{-1}\ \ \
\mbox{and}\ \ \ \lim_{j\to\infty}x_j=
\alpha\,\tau\,\log\,\omega^{-1}.
$$
Then
$$
\prod_{j=1}^s\left(m_0+\frac{m_1}{(j+1)^{x_j}}\right)\le
\prod_{j=1}^s\left(m_0+\sum_{k=1}^\infty
m_k\omega^{\tau\,a_j\,k^{b_j}}\right)  \le
\prod_{j=1}^s\left(m_0+\frac{C_{\tau}}{(j+1)^{x_j}}\right).
$$
Hence, \eqref{ptspt} holds iff $m_0=1$ and $\lim_jx_j\ge1$.
Indeed, $m_0=1$ is clear because otherwise we have an exponential
dependence on $s$. For $m_0=1$, let $\beta\in\{m_1,C_{\tau}\}$.
Then
$$
\prod_{j=1}^s\left(m_0+\frac{\beta}{(j+1)^{x_j}}\right)=
\exp\left(\sum_{j=1}^s\log(1+\beta\,(j+1)^{-x_j})\right).
$$
Furthermore,
$$
\sum_{j=1}^s\log(1+\beta\,(j+1)^{-x_j})=
\Theta\left(\sum_{j=1}^s(j+1)^{-x_j}\right),
$$
with the factors in the big $\Theta$ notation independent of $s$
and $j$.

Suppose that $\alpha=0$. Then $\lim_jx_j=0$ for all $\tau$. This
means for all $\delta\in(0,1)$ there is an integer
$j(\delta,\tau)$ such that $x_j\le\delta$ for all $j\ge
j(\delta,\tau)$, and
$\sum_{j=1}^s(j+1)^{-x_j}=\Theta(s^{1-\delta})$. Hence
$$
\prod_{j=1}^s\left(m_0+\frac{\beta}{(j+1)^{x_j}}\right) \ \
\mbox{as
  well as}\ \
\left(\sum_{k=1}^\infty\lambda_{s,k}^{\tau}\right)^{1/\tau}
$$
is exponential in $s^{1-\delta}$. This means that \eqref{ptspt}
does not hold for any positive $\tau$ and non-negative $q$. Hence,
we do not have PT.

Suppose now that $\alpha>0$. Then $\lim_jx_j>1$ for
$\tau>(\alpha\,\log\,\omega^{-1})^{-1}$. This implies that
$$
\sum_{j=1}^s(j+1)^{-x_j}
 \ \ \mbox{as
  well as}\ \
\left(\sum_{k=1}^\infty\lambda_{s,k}^{\tau}\right)^{1/\tau}
$$
is uniformly bounded in $s$. Hence, \eqref{ptspt} holds for $q=0$
and we have SPT with the exponent $p^*\le
2/(\alpha\,\log\,\omega^{-1})$. For
$\tau<(\alpha\,\log\,\omega^{-1})^{-1}$, the series
$\sum_{j=1}^s(j+1)^{-x_j}$ is of order at least $\log\,s$ and
\eqref{ptspt} may hold only for $q>0$. This contradicts SPT. Hence
$p^*\ge 2/(\alpha\,\log\,\omega^{-1})$, which completes the proof.
\end{proof} 
\end{itemize}

\subsection{New Notions of Tractability}

We now turn to new notions of tractability which correspond to the
standard notions of tractability for the pair
$(s,1+\log\,\eps^{-1})$ instead of the pair $(s,\eps^{-1})$. To
distinguish between the standard and new notions of tractability,
we add the prefix EC (exponential convergence) when we consider
the new notions. As before, we study the new notions of
tractability for the approximation problem
$\APP=\{\APP_s\}_{s\in\NN}$ for general parameters
$\bsa,\bsb,\bsm$ and $\omega$.
\begin{itemize}
\item {\bf EC-$(t_1,t_2)$-Weak Tractability}

We say that $\APP$ is EC-$(t_1,t_2)$-WT iff
$$
\lim_{s+\eps^{-1}\to\infty} \frac{\log\,n(\eps,\APP_s)}{s^{t_1}+
[\log\,\eps^{-1}]^{t_2}}=0.
$$
Obviously, EC-$(t_1,t_2)$-WT implies $(t_1,t_2)$-WT. For $t_1=1$
and $t_2>1$, this notion was introduced and studied in
\cite{PP14}.

\begin{thm}\label{thm6}

\qquad

\qquad $\APP$ is {\rm EC-$(t_1,t_2)$-WT} for the parameters
$\bsa,\bsb,\bsm$ and $\omega$ \ \ iff \ \ $t_1>1$, 
 or $t_2>1$ and $m_0=1$.
\end{thm}  
\begin{proof}
Suppose that $\APP$ is EC-$(t_1,t_2)$-WT for the parameters
$\bsa,\bsb,\bsm$ and $\omega$.
Then for fixed 
$\varepsilon\in (0,1)$ we obtain from \eqref{lem1p1} of Lemma~\ref{lem1} that 
$$
0=\lim_{s \rightarrow \infty} 
\frac{\log n(\varepsilon,\APP_s)}{s^{t_1} + 
[\log \varepsilon^{-1}]^{t_2}} \ge \lim_{s \rightarrow \infty} 
\frac{s \log m_0}{s^{t_1} + 
[\log \varepsilon^{-1}]^{t_2}}=\lim_{s \rightarrow \infty} s^{1-t_1}
\log m_0.
$$ 
Hence, we conclude that $t_1>1$ or that $m_0=1$.
For $t_1\le1$ and $m_0=1$, it 
remains to show that $t_2>1$. As in \cite[p.~609]{PP14} 
we find that for $\varepsilon_s 
\in (\lambda^{(\lfloor s/2\rfloor +1)/2},\lambda^{\lfloor
  s/2\rfloor/2})=:L_s$,
where $\lambda:=\omega^{a_1}$, 
we have $n(\varepsilon_s,\APP_s) \ge 2^{\lfloor s/2\rfloor}$. Then 
\begin{eqnarray*}
0 =\lim_{s \rightarrow \infty \atop \varepsilon_s \in L_s} 
\frac{\log n(\varepsilon_s,\APP_s)}{s^{t_1}+
[\log \varepsilon_s^{-1}]^{t_2}} 
\ge \lim_{s \rightarrow \infty} \frac{\lfloor s/2 \rfloor 
\log 2}{s^{t_1}+ \left(\frac{\lfloor s/2\rfloor +1}{2} 
\log \lambda^{-1} \right)^{t_2}}. 
\end{eqnarray*}
This can only hold 
if $t_2 >1$.

Suppose now that $t_1>1$. {}From \eqref{generalbound} we have for
small $\eps$,
$$
\frac{\log\,n(\eps,\APP_s)}{s^{t_1}+[\log\,\eps^{-1}]^{t_2}}=
\mathcal{O}\left(
\frac{s\,\log\,\log\,\eps^{-1}}
{s^{t_1}+[\log\,\eps^{-1}]^{t_2}}\right).
$$
Let $y=\max(s^{t_1},[\log\,\eps^{-1}]^{t_2})$. Then 
$$
\frac{s\,\log\,\log\,\eps^{-1}}
{s^{t_1}+[\log\,\eps^{-1}]^{t_2}}\le\frac{y^{1/t_1}\,\log\,y}{t_2\,y}.
$$
Clearly, this goes to zero as $s+\eps^{-1}$ approaches infinity
since $t_1>1$ and $t_2>0$. Hence, we have EC-$(t_1,t_2)$-WT, as
claimed.

Suppose now that $t_2>1$ and $m_0=1$. From the proof of Theorem~\ref{thm2}, \eqref{prthm2a}, \eqref{prthm2b} and \eqref{prthm2c} we obtain 
$$\log n(\varepsilon,\APP_s) \le C\log\varepsilon^{-1} \log s + C_3 \log \varepsilon^{-1} \log \log \varepsilon^{-1}$$ with suitable constants $C,C_3>0$. 
Since $t_1>0$ and $t_2>1$ and using the same argument for $y=\max(s^{t_1},[\log\,\eps^{-1}]^{t_2})$ as above,
it follows that $$\lim_{s+\varepsilon^{-1}\rightarrow \infty} \frac{\log n(\varepsilon,\APP_s)}{s^{t_1}+[\log \varepsilon^{-1}]^{t_2}} =0.$$
Hence, we have EC-$(t_1,t_2)$-WT. This completes the proof.
\end{proof}


\item {\bf EC-Weak and EC-Uniform Weak Tractability}

EC-weak tractability (EC-WT) corresponds to EC-$(1,1)$-WT. 
EC-uniform weak tractability (EC-UWT) means that
EC-$(t_1,t_2)$-WT holds for all $t_1,t_2\in(0,1]$.   
Clearly, EC-WT implies WT, and EC-UWT implies UWT.

\begin{thm}\label{thm7}
\qquad
\begin{itemize}
\item
$\APP$ is {\rm EC-WT}  \ \ \ iff \ \ \ \ \ $m_0=1$\ \
\mbox{and}\ \ $\lim_{j\to\infty}a_j=\infty$,
\item
$\APP$ is {\rm EC-UWT}  \ \ \ iff \ \ \ $m_0=1$\ \
\mbox{and}\ \ $\lim_{j\to\infty}\ \frac{\log\,a_j}{\log\,j}=\infty$.
\end{itemize}
\end{thm}
\begin{proof}
We first assume that EC-WT or EC-UWT holds. 
Since EC-WT implies WT and EC-UWT implies UWT, Theorem~\ref{thm3}
implies that $m_0=1$. For $t_1,t_2\in(0,1]$, let
$$ 
z_{s,t_1,t_2}=\frac{\log\,n(\eps,\APP_s)}{s^{\,t_1}+[\log\,\eps^{-1}]^{t_2}}.
$$
Let
$\delta>0$ and take $x(\eps)=(1+\delta)(a_1+\cdots+a_s)$. Due to
the definition \eqref{functionx} of $x(\eps)$ this means that
$$
\log\,\eps^{-1}=\frac{\log\,\omega^{-1}}{2}\,(1+\delta)\,(a_1+a_2+\cdots+a_s).
$$
{}From \eqref{lem1p2} of Lemma~\ref{lem1} we have
$$
z_{s,t_1,t_2}
\ge
\frac{s\,\log(1+m_1)}{s^{\,t_1}+[\log\,\eps^{-1}]^{t_2}}\ge\frac{\log\,2}
{s^{\,t_1-1}+ [\tfrac12(1+\delta)\,\log\,\omega^{-1}]^{t_2}y_s},
$$
where
$$
y_s=\frac{(a_1+a_2+\cdots+a_s)^{t_2}}{s}\le \frac{(s\,a_s)^{t_2}}{s}
=\frac{a_s^{t_2}}{s^{\,1-t_2}}.
$$
Then $\lim_sz_{s,t_1,t_2}=0$ implies that $\lim_sy_s=\infty$, which in
turn implies that
\begin{equation}\label{fun}
\lim_{s\to\infty}\frac{a_s^{t_2}}{s^{\,1-t_2}}=
\lim_{s\to\infty}\frac{a_s}{s^{\,t_2^{-1}-1}}=\infty.
\end{equation}
If we have EC-WT then $\lim_sz_{s,1,1}=0$ and 
$\lim_ja_j=\infty$, as claimed. If we have EC-UWT then, in particular,
$\lim_sz_{s,1,t_2}=0$ for all positive $t_2$. Then~\eqref{fun} yields
there is a number
$s^*=s^*(t_2)$ such that 
$$
a_s\ge s^{\,t_2^{-1}-1}\ \ \ \ \mbox{for all}\ \ s\ge s^*(t_2),
$$
or equivalently
$$
\frac{\log\,a_s}{\log\,s}\ge \frac{1-t_2}{t_2}
\ \ \ \ \mbox{for all}\ \ s\ge s^*(t_2).
$$
Since $t_2$ can be arbitrarily close to zero this implies that
$$
\lim_{s\rightarrow \infty}\frac{\log\,a_s}{\log\,s}=\infty,
$$
as claimed.  
 
We now prove that $m_0=1$ and $\lim_ja_j=\infty$ imply EC-WT.
For any positive $\eta$ we compute the $\eta$ powers of the
eigenvalues $\lambda_{s,k}$ of the operator $W_s$. We have
$$
n\lambda_{s,n}^{\eta}\le\sum_{k=1}^\infty\lambda_{s,k}^{\eta}
=\prod_{j=1}^s\left(1+\sum_{k=1}^\infty
  m_k\,\omega^{\eta\,a_j\,k^{b_j}}\right)
\le \prod_{j=1}^s\left(1+ C_\eta\,\omega^{\eta\,a_j}\right),
$$
due to \eqref{ccc}. Hence,
$$
\lambda_{s,n}\le \frac{\prod_{j=1}^s\left(1+
  C_{\eta}\,\omega^{\eta\,a_j}\right)^{1/\eta}}{n^{1/\eta}}.
$$
Then \eqref{9} yields
$$
n(\eps,\APP_s)\le \frac{\prod_{j=1}^s\left(1+
  C_{\eta}\,\omega^{\eta\,a_j}\right)}{\eps^{2\eta}}.
$$
Since $\log(1+x)\le x$ for positive $x$, we conclude
$$
\log\,n(\eps,\APP_s)\le 2\eta\,\log\,\eps^{-1}\ +
C_{\eta}\,\sum_{j=1}^s\omega^{\eta\,a_j}.
$$
Observe that $\lim_ja_j=\infty$ implies
$\lim_j\omega^{\eta\,a_j}=0$ and
$\lim_s\sum_{j=1}^s\omega^{\eta\,a_j}/s=0$. Therefore
$$
\limsup_{s+\eps^{-1}\to \infty}
\frac{\log\,n(\eps,\APP_s)}{s+\log\,\eps^{-1}}\le 2\eta.
$$
Since $\eta$ can be arbitrarily small this proves that
$$
\lim_{s+\eps^{-1}\to \infty}
\frac{\log\,n(\eps,\APP_s)}{s+\log\,\eps^{-1}}=0.
$$
Hence EC-WT holds. 

Finally, we prove that $m_0=1$ and $\lim_s(\log\,a_s)/\log\,s=\infty$
imply EC-UWT. 
For $x(\eps)>a_1$, \eqref{lem1p3} of Lemma~\ref{lem1}  yields
$$
\log\,n(\eps,\APP_s)\le
\sum_{j=1}^{j(\eps)}\log\left(1+
\mmax\left(\frac{x(\eps)}{a_j}
\right)^{1/b_j}\right).
$$
Let $b_0=\inf_jb_j>0$ and 
$$
\alpha=\mmax\left(\frac2{a_1\,\log\,\omega^{-1}}\right)^{1/b_0}.
$$
{}From the definition \eqref{functionx} of $x(\eps)$ we then have
$$
\mmax\left(\frac{x(\eps)}{a_j}
\right)^{1/b_j}\le \alpha [\log\,\eps^{-1}]^{1/b_0},
$$
and
$$
\log\,n(\eps,\APP_s)
\le j(\e)\,\log\left(1+\alpha
[\log\,\eps^{-1}]^{1/b_0}\right)=\mathcal{O}
\left(j(\eps)\,\log\,\log\,\eps^{-1}\right).
$$
We now estimate $j(\eps)$ using the assumption that
$\lim_j(\log\,a_j)/\log\,j=\infty$. We know that for all positive
$\tau$ there is a number $j_{\tau}$ such that 
$$
a_j\ge j^{\tau}\ \ \ \ \mbox{for all}\ \ \ j\ge j_{\tau}.
$$
This implies that 
$$
j(\eps)\le \max\left(j_{\tau}-1, x(\eps)^{1/\tau}\right)=
\mathcal{O}\left([\log\,\eps^{-1}]^{1/\tau}\right).
$$
Therefore
$$
\log\,n(\eps,\APP_s)=\mathcal{O}\left([\log\,\eps^{-1}]^{1/\tau}\,
\log\,\log\,\eps^{-1}\right).
$$
We stress that the factors in the big $\mathcal{O}$ notation do not
depend on $s$. Then for any positive $t_1,t_2\in(0,1]$ we take 
$\tau>1/t_2$, and conclude 
$$
\lim_{s+\eps^{-1}\to\infty}\frac{\log\,n(\eps,\APP_s)}{s^{\,t_1}
+[\log\,\eps^{-1}]^{t_2}}=0.
$$
This proves EC-UWT, and completes the proof.
\end{proof}

If we compare Theorems~\ref{3} and~\ref{7}, we see that the
assumption $m_0=1$ is always needed. However, WT
holds for all $\bsa=\{a_j\}_{j\in\NN}$, whereas EC-WT requires that
$\lim_ja_j=\infty$. 
Similarly, UWT holds for all $\bsa=\{a_j\}_{j\in\NN}$,
whereas EC-UWT requires that $\lim_j(\log\,a_j)/\log\,j=\infty$.
Hence, $a_j$'s may go to infinity arbitrarily slowly for EC-WT,
whereas they must go to infinity faster than polynomially to get
EC-UWT.  
It seems interesting that WT, UWT, EC-WT and EC-UWT do not
depend on $\bsb$,~$\bsm$ (with $m_0=1$) and $\omega$.

\item {\bf EC-Quasi-Polynomial Tractability}

$\APP$ is EC-QPT if there are positive $C$ and $t$ such that
$$
n(\eps,\APP_s)\le
C\,\exp\left(t\,(1+\log\,s)(1+\log\,(1+\log\,\eps^{-1}))\right) \
\ \ \ \mbox{for all}\ \ s\in\NN,\ \eps\in(0,1).
$$
The infimum of $t$ satisfying the bound above is denoted by $t^*$,
and is called the exponent of EC-QPT. Obviously, EC-QPT implies
EC-WT.

Observe that
\begin{eqnarray*}
\exp\left(t\,(1+\log\,s)(1+\log\,(1+\log\,\eps^{-1}))\right)&=&
[{\rm e}\,s]^{\ t\,(1+\log\,(1+\log\,\eps^{-1}))}\\
&=&[{\rm e}(1+\log\,\eps^{-1})]^{\ t\,(1+\log\,s)}.
\end{eqnarray*}
We will sometimes use these equivalent formulations to establish
EC-QPT.

\begin{thm}\label{thm8}

\qquad

\quad $\APP$ is {\rm EC-QPT}\ \ iff \ \
$$
m_0=1,\ \ \ B^*:=\sup_{s\in \NN}\
\frac{\sum_{j=1}^sb_j^{-1}}{1+\log\,s}<\infty,\ \ \mbox{ and }\
\alpha:=\liminf_{j\to\infty}\ \frac{(1+\log\,j)\,\log\,a_j}{j}>0.
$$ 
If this holds then the exponent of {\rm EC-QPT} satisfies
$$
t^*\in \left[\max\left(B^*,\frac{\log\,(1+m_1)}{\alpha}\right),
B^*+\frac{\log\,(1+\mmax)}{\alpha}\right].
$$
In particular, if $\alpha=\infty$ then $t^*=B^*$. 
\end{thm}
\begin{proof}
We first prove that EC-QPT implies the conditions on $m_0,B^*$ and
$\alpha$. Since EC-QPT yields EC-WT, we have $m_0=1$. To prove
that $B^*<\infty$, we relate EC-QPT to EXP. {}From
$$
n=n(\eps,\APP_s)\le
C\,\exp\left(t\,(1+\log\,s)(1+\log\,(1+\log\,\eps^{-1}))\right)
$$
we conclude that
$$
e(n,\APP_s)\le \eps\le {\rm e}\cdot\exp\left(-\frac1{{\rm
e}}\left(\frac{n}{C}\right)^{(t(1+\log\,s))^{-1}}\right).
$$
Due to Theorem~\ref{thm1}, EXP holds with the exponent
$1/B_s=1/\sum_{j=1}^sb_j^{-1}$. Therefore $1/(t(1+\log\,s))\le
1/B_s$ and
$$
\frac{B_s}{1+\log\,s}\le t\ \ \ \ \mbox{for all}\ \ s\in\NN.
$$
Hence, $B^*<\infty$ and $t\ge B^*$, as claimed.

To prove that $\alpha>0$, we proceed similarly as 
for EC-WT and EC-UWT.
That
is, for a positive $\delta$, we take
$$
x(\eps)=(1+\delta)(a_1+\dots+a_s)\le (1+\delta)\,s\,a_s.
$$
Now \eqref{lem1p2} of Lemma~\ref{lem1} yields
$$
s\,\log\,(1+m_1)\le\,\log\,n(\eps,\APP_s)\le \log\,C\ +\
t(1+\log\,s)\,(1+\log\,(1+\log\,\eps^{-1})),
$$
where
$$
1+\log\,(1+\log\,\eps^{-1}))\le 1+\log\,
\left(1+\frac{(1+\delta)\log\,\omega^{-1}}2\,s\,a_s\right).
$$
For large $s$, this proves that
$$
\frac{t\,(1+\log\,s)\,\log\,a_s}{s}\ge \log\,(1+m_1)\ +\ o(s).
$$
Hence, $\alpha>0$ and $t\ge \log(1+m_1)/\alpha$, as claimed.

We now prove that $m_0=1$, $B^*<\infty$ and $\alpha>0$ imply
EC-QPT.

{}From $m_0=1$ and Lemma~\ref{lem1} we have $n(\eps,\APP_s)=1$ for
$x(\eps)\le a_1$, whereas for $x(\eps)>a_1$, we have
\begin{eqnarray*}
n(\eps,\APP_s)&\le&\prod_{j=1}^{\min(s,j(\eps))}\left(1+\mmax\left(
\frac{x(\eps)}{a_j}\right)^{1/b_j}\right)\\
&\le&(1+\mmax)^{\min(s,j(\eps))}\,
\prod_{j=1}^{\min(s,j(\eps))}\left(\frac{x(\eps)}{a_j}\right)^{1/b_j}.
\end{eqnarray*}
{}From the definition  \eqref{functionx} of $x(\eps)$, we get
\begin{equation}\label{1212}
n(\eps,\APP_s)\le (1+\mmax)^{\min(s,j(\eps))}\, C_{s,\eps}\, [{\rm
e}\,\log\,\eps^{-1}]^{\, \sum_{j=1}^{\min(s,j(\eps))}b_j^{-1}},
\end{equation}
where
$$
C_{s,\eps}=\prod_{j=1}^{\min(s,j(\eps))} \left(\frac{2}{a_j\,{\rm
e}\,\log\,\omega^{-1}}\right)^{1/b_j}.
$$
Note that \eqref{1212} holds for all $s\in\NN$ and all
$\eps\in(0,1)$ if we take $j(\eps)=0$ for $x(\eps)\le a_1$.

We now use the assumption that $\alpha>0$. This means that for any
$\delta\in(0,\alpha)$ there is an integer~$j_\delta$ such that
$$
a_j\ge \exp\left(\frac{\delta\,j}{1+\log\,j}\right)\ \ \ \
\mbox{for all}\ \ j\ge j_\delta.
$$
This means that $\lim_ja_j=\infty$, and this convergence is almost
exponential in~$j$.

We turn to $j(\eps)$ defined by~\eqref{13}. Now $j(\eps)$ goes to
infinity  as $\eps$ approaches zero. For $\log\,x(\eps)\ge
\delta$, i.e., for $\eps\le \omega^{{\rm e}^\delta/2}$, we have
$$
j(\eps)\le \max(j_\delta, J(\eps)),
$$
where $J(\eps)$ is a solution of the nonlinear equation
\begin{equation}\label{eqJeps}
\frac{\log\,x(\eps)}{\delta}=\frac{J(\eps)}{1+\log\,J(\eps)}.
\end{equation}
The solution is unique since the function $y/(1+\log\,y)$ is
increasing for $y\ge1$.

Let $a(\varepsilon)=(\log x(\eps))/\delta$. Then we have from \eqref{eqJeps} that $J(\eps)=a(\eps)(1+\log J(\eps))$. 
Now we write $J(\eps)$ in the form $$J(\eps)=(1+f(\eps)) a(\eps) \log a(\eps),$$ 
where $f(\eps)$ is given by 
$$f(\eps)=\frac{1+\log(1+\log J(\eps))}{\log a(\eps)}=\frac{1+\frac{1}{\log(1+\log J(\eps))}}{\frac{\log J(\eps)}{\log(1+\log J(\eps))}-1}=o(1)\ \ \mbox{ for }\ \eps \rightarrow 0.$$ 
Hence we have
\begin{eqnarray}\label{1213}
J(\eps)& =& (1+o(1)) a(\eps) \log a(\eps) \nonumber \\ 
       & =& \frac{1+o(1)}{\delta}\,[\log\,x(\eps)]\ \log\,\frac{\log\, x(\eps)}{\delta}\nonumber \\
       & =& \frac{1+o(1)}{\delta}\,[\log\,\log\,\eps^{-1}]\ \log\,\log\,\log\,\eps^{-1}.
\end{eqnarray}

We turn to \eqref{1212}. Note that $\lim_ja_j=\infty$ implies that
only a finite number of factors in $C_{s,\eps}$ is larger than
one. Therefore
$$
C_{s,\eps}\le C_1:=\sup_{s\in\NN,\ \eps\in(0,1)}
\prod_{j=1}^{\min(s,j(\eps))} \left(\frac{2}{a_j\,{\rm
e}\,\log\,\omega^{-1}}\right)^{1/b_j} <\infty.
$$
Furthermore, from the assumption $B^*<\infty$ we have
$$
\sum_{j=1}^{\min(s,j(\eps))}b_j^{-1}=
\frac{\sum_{j=1}^{\min(s,j(\eps))}b_j^{-1}}{1+\log\,s}\,(1+\log\,s)\le
  B^*\,(1+\log\,s).
$$
Therefore we can rewrite~\eqref{1212} as
\begin{equation}\label{1214}
n(\eps,\APP_s)\le (1+\mmax)^{\min(s,j(\eps))}\,C_1\,[{\rm
e}\,(1+\log\,\eps^{-1})]^{B^*\,(1+\log\,s)}.
\end{equation}

We now analyze the first factor $\beta:=(1+\mmax)^{\min(s,j(\eps))}$
in~\eqref{1214}. Let $s^*\in\NN$ and $\eps^*\in(0,1)$ be
arbitrary. Note that for $s\le s^*$ or for $\eps\in[\eps^*,1]$ we
have
$$
\beta \le (1+\mmax)^{\max(s^*,\,j(\eps^*))}=:
C_2<\infty.
$$
Hence, without loss of generality we can consider
$$
s>s^*\ \ \ \mbox{and}\ \ \ \eps\in(0,\eps^*).
$$
We now choose $s^*$ such that $s^*\ge j_{\delta}$. For any
positive $\eta\in(0,1)$ we choose a positive $\eps^*$ such that
for all $\eps\in(0,\eps^*)$ we have
\begin{eqnarray}
\log\,\log\,\log\,\eps^{-1}&\ge&\frac{\delta}{1-\eta},\label{eps1}\\
\frac{\delta\,J(\eps)} {\log\,\log\,\eps^{-1}\
\log\,\log\,\log\,\eps^{-1}}
&\in&[1-\eta,1+\eta],\label{eps2}\\
\frac{(1+\eta)}
{\delta\left(1+\frac{\log(1-\eta)/\delta\,+\,\log\,\log\,\log\,\log\,\eps^{-1}}
{\log\,\log\,\log\,\eps^{-1}}\right)}&\le&
\frac{1+2\eta}{\delta}.\,\label{eps3}
\end{eqnarray}
Observe that such a positive $\eps^*$ exists since~\eqref{eps1}
clearly holds for small $\eps$, whereas~\eqref{eps2} holds due
to~\eqref{1213}, and~\eqref{eps3} holds since the limit of the
left hand side, as $\eps\to\ 0$, is $(1+\eta)/\delta$ which is smaller than the
right hand side.

We are ready to estimate
$$
\beta=[{\rm e}\,s]^{y_{s,\eps}}= [{\rm
e}\,(1+\log\,\eps^{-1})]^{z_{s,\e}},
$$
where
\begin{eqnarray*}
y_{s,\eps}&=& \frac{\min(s,j(\eps))\,\log\,(1+\mmax)}{\log\,({\rm e} s)},\\
z_{s,\eps}&=& \frac{\min(s,j(\eps))\,\log\,(1+\mmax)}{\log\,({\rm
e}\,(1+\log\,\eps^{-1}))}.
\end{eqnarray*}

We consider two cases depending on 
whether $s$ or $J(\eps)$ is
larger.

{\bf Case 1.} Assume that $s\le J(\eps)$.

Note that the function $y/(1+\log\,y)$ is an increasing function
of $y\in[1,\infty)$. Therefore
$$
\frac{s}{1+\log\,s}\le \frac{J(\eps)}{1+\log\, J(\eps)}.
$$
Due to~\eqref{eps2} and~\eqref{eps3},
\begin{eqnarray*}
\frac{J(\eps)}{1+\log\, J(\eps)}&\le&
\frac{(1+\eta)\,\log\,\log\,\eps^{-1}\
\log\,\log\,\log\,\eps^{-1}}
{\delta(1+\log\frac{1-\eta}{\delta}+\log\,\log\,\log\,\eps^{-1}+
\log\,\log\,\log\,\log\,\eps^{-1})}\\
&\le&\frac{1+2\eta}{\delta}\,\log\,\log\,\eps^{-1}\le
 \frac{1+2\eta}{\delta}\,(1+\log\,(1+\log\,\eps^{-1})).
\end{eqnarray*}
Hence,
\begin{eqnarray*}
y_{s,\eps}&\le& \frac{s}{1+\log\,s}\,\log\,(1+\mmax)\le
\frac{J(\eps)}{1+\log\,J(\eps)}\,\log\,(1+\mmax)\\
&\le&
\frac{(1+2\eta)\,\log(1+\mmax)}{\delta}\,\left(1+\log(1+\log\,\eps^{-1})\right).
\end{eqnarray*}
This yields
$$
\beta\le [{\rm e}
s]^{\delta^{-1}(1+2\eta)\log\,(1+\mmax)\,(1+\log\,(1+\log\,\eps^{-1}))}
$$
which can be equivalently written as
$$
\beta \le \exp\left(
\left[\frac{1+2\eta}{\delta}\,\log(1+\mmax)\right]\,
(1+\log\,s)(1+\log\,(1+\log\,\eps^{-1}))\right).
$$
This and~\eqref{1214} yield EC-QPT with $t\le
B^*+\delta^{-1}(1+2\eta)\,\log\,(1+\mmax)$. Since $\delta$ can be
arbitrarily close to $\alpha$ and $\eta$ can be arbitrarily small,
we conclude that the exponent of EC-QPT in this case satisfies
$$
t\le B^*+\frac{\log\,(1+\mmax)}{\alpha}.
$$

{\bf  Case 2.}  Assume that $s>J(\eps)$.

Then $s>\delta^{-1}\,(1-\eta)\,[\log\,\log\,\eps^{-1}]\
\log\,\log\,\log\,\eps^{-1}\ge \log\,\log\,\eps^{-1}$ due
to~\eqref{eps1} and~\eqref{eps2}. Hence, $\log\,s\ge
\log\,\log\,\log\,\eps^{-1}$. We now estimate $z_{s,\eps}$. Assume that $\eps>0$ is small enough such that $j(\eps)\le J(\eps)$. Then we
have
\begin{eqnarray*}
z_{s,\eps}&\le&\frac{j(\eps)\,\log\,(1+\mmax)}
{1+\log\,(1+\log\,\eps^{-1})}\le \frac{J(\eps)\,\log\,(1+\mmax)}
{1+\log\,(1+\log\,\eps^{-1})}\\
&\le&\frac{(1+\eta)\,\log\,(1+\mmax)}{\delta}\,
\frac{\log\,\log\,\eps^{-1}\
\log\,\log\,\log\,\eps^{-1}}{1+\log\,(1+\log\,\eps^{-1})}\\
&\le& \frac{(1+\eta)\,\log\,(1+\mmax)}{\delta}\,
\log\,\log\,\log\,\eps^{-1}\\
&\le&
 \frac{(1+\eta)\,\log\,(1+\mmax)}{\delta}\,(1+\log\,s).
\end{eqnarray*}
Hence we have
$$
\beta\le\exp\left(\left[\frac{1+\eta}{\delta}\,
\log\,(1+\mmax)\right]\,(1+\log\,s)(1+\log\,(1+\log\,\eps^{-1}))\right),
$$
and the rest of the proof goes like in Case 1. This completes the
proof.
\end{proof}

We compare Theorems \ref{thm4} and \ref{thm8}. The assumption
$m_0=1$ is needed for both QPT and EC-QPT. However, QPT holds for
all $\bsa$ and $\bsb$, whereas for EC-QPT we need to assume that
$B^*<\infty$ and $\alpha>0$. This means that $b_j$ must go to
infinity roughly at least 
like $j$, and $a_j$ must go to infinity almost
exponentially fast.

\item {\bf EC-Polynomial and EC-Strong Polynomial Tractability}

$\APP$ is EC-PT iff there are positive $C,p$ and $q\ge0$ such that
$$
n(\eps,\APP_s)\le C\,s^{\,q}\,(1+\log\,\eps^{-1})^{p} \ \ \ \
\mbox{for all}\ \ s\in\NN,\ \eps\in(0,1).
$$
$\APP$ is EC-SPT if the last bound holds with $q=0$, and then the
infimum of $p$ is denoted by $p^*$, and is called the exponent of
EC-SPT.

\begin{thm}\label{thm9}

\quad

\quad $\APP$ is {\rm EC-PT}\ \ iff \ \ $\APP$ is {\rm EC-SPT}  \  \ iff

$$
m_0=1,\ \ B:=\sum_{j=1}^\infty \frac1{b_j}<\infty \ \ \mbox{and}\
\ \alpha^*=\liminf_{j\to \infty}\ \frac{\log\,a_j}j>0.
$$
If these conditions hold then the exponent of {\rm EC-SPT}
satisfies
$$
p^*\in\left[\max\left(B,\frac{\log(1+m_1)}{\alpha^*}\right),
  B+\frac{\log(1+\mmax)}{\alpha^*}\right].
$$
In particular, if $\alpha^*=\infty$ then $p^*=B$.
\end{thm}
\begin{proof}
We prove that EC-PT implies $m_0=1$, $B<\infty$ and $\alpha^*>0$,
and then that $m_0=1$, $B<\infty$ and $\alpha^*>0$ imply EC-SPT
and find bounds on the exponent of EC-SPT.

EC-PT implies EC-WT and therefore $m_0=1$. It is easy to show that
EC-PT implies UEXP. Indeed, the bound on EC-PT yields that
$$
e(n,\APP_s)\le {\rm e} \cdot {\rm e}^{-((n-1)/(C\,s^{\,q}))^{1/p}}
\ \ \ \ \mbox{for all} \ \ n\in\NN.
$$
Hence, UEXP holds and the exponent of UEXP is at least $1/p$. Then
Theorem~\ref{thm1} implies that $B<\infty$, and $p\ge B$.

To prove that $\alpha^*>0$, we proceed similarly as for EC-WT.
That is, for $\delta>0$ we take
$x(\eps)=(1+\delta)(a_1+\cdots+a_s)$ and then \eqref{lem1p2} of
Lemma~\ref{lem1} and the bound on EC-PT yield
$$
(1+m_1)^s\le n(\eps,\APP_s)\le
C\,s^{\,q}\,\left(1+\frac{(1+\delta)\,\log\,\omega^{-1}}2(a_1+\cdots+a_s)
\right)^p.
$$
Since $a_1\le a_2\le\cdots$, this implies that
$$
sa_s\ge a_1+\cdots +a_s\ge \frac2{(1+\delta)\,\log\,\omega^{-1}}
\left[\left(\frac{(1+m_1)^s}{C\,s^{\,q}}\right)^{1/p}-1\right].
$$
Hence,
$$
\alpha^*=\liminf_{s\to \infty}\,\frac{\log\,
  a_s}{s}\,\ge\,\frac{\log(1+m_1)}{p}>0,
$$
as claimed. This also shows that $p\ge \log(1+m_1)/\alpha^*$.

This reasoning also holds for all $p$ for which EC-SPT holds.
Therefore the exponent~$p^*$ of EC-SPT is at least $p^*\ge
\log(1+m_1)/\alpha^*$. Furthermore, $p^*$ cannot be smaller then
the reciprocal of the exponent of UEXP, so that $p^*\ge B$. This
proves the lower bound on $p^*$.

We now assume that $m_0=1$, $B<\infty$ and $\alpha^*>0$. Note that
$\alpha^*>0$ means that $a_j$ are exponentially large in $j$ for
large $j$. Indeed, for $\delta\in(0,\alpha^*)$ there is an integer
$j^*_\delta$ such that
$$
a_j\ge \exp(\delta\,j)\ \ \ \ \mbox{for all}\ \ j\ge j^*_\delta.
$$
This yields that for $j(\eps)$ defined by \eqref{13} we have
$$
j(\eps)\le \max\left(j^*_\delta,
\frac{\log\,x(\eps)}{\delta}\right).
$$
For $x(\eps)\le a_1$, \eqref{lem1p1} of Lemma~\ref{lem1} states
that $n(\eps,\APP_s)=1$, whereas for $x(\eps)>a_1$, \eqref{lem1p3} of Lemma~\ref{lem1} yields
\begin{eqnarray*}
n(\eps,\APP_s)&\le& \prod_{j=1}^{\min(s,j(\eps))}
\left(1+\mmax\left(\frac{x(\eps)}{a_j}\right)^{1/b_j}\right)\\
&\le&
\left[\prod_{j=1}^{\min(s,j(\eps))}\left(\frac{x(\eps)}{a_j}\right)^{1/b_j}
\right]\,(1+\mmax)^{\min(s,j(\eps))}\\
&\le&\left(\frac{x(\eps)}{a_1}\right)^B\,\max\left((1+\mmax)^{j^*_\delta},
[x(\eps)]^{(\log(1+\mmax))/\delta}\right).
\end{eqnarray*}
Since $x(\eps)=\Theta(\log\,\eps^{-1})$ we obtain EC-SPT with
$p\le B+(\log(1+\mmax))/\delta$. Taking $\delta$ arbitrarily close
to $\alpha^*$, we obtain that the exponent of EC-SPT is at most
$$
p^*\le B+\frac{\log(1+\mmax)}{\alpha^*}.
$$
This completes the proof.
\end{proof}

We now compare Theorems~\ref{5} and~\ref{9} for SPT and EC-SPT. In
both cases, we have $m_0=1$. However, the conditions on $\bsa$ and
$\bsb$ are quite different. SPT holds for all~$\bsb$, whereas for
EC-SPT we must assume that $B<\infty$, i.e., $b_j$ must go to
infinity at least like~$j$. The conditions on $\bsa$ are even more
striking. SPT holds for $a_j$ going to infinity quite slowly like
$\log\,j$, whereas EC-SPT requires that $a_j$ goes exponentially
fast to infinity with~$j$.
\end{itemize}

\section{Summary}\label{sec_sum}

In the following table we summarize the tractability
results. We tabulate the various notions of tractability with their
corresponding ``if and only if'' conditions:

\begin{center}
\renewcommand*\arraystretch{2}
\begin{tabular}{|l|l|}
\hline
Tractability notion & iff-conditions\\
\hline \hline
$(t_1,t_2)$-WT 
& $t_1 >1$ or $m_0=1$ \\
\hline
WT, UWT, QPT & $m_0=1$\\
\hline
PT, SPT & $m_0=1$, and
$\lim_{j
} \frac{a_j}{\log j}>0$\\
\hline
EC-$(t_1,t_2)$-WT 
& $t_1 >1$, or $t_2>1$ and $m_0=1$ \\
\hline
EC-WT & $m_0=1$, and
$\lim_{j
} a_j=\infty$\\
\hline
EC-UWT & $m_0=1$, and $\lim_{j}\frac{\log\,a_j}{\log\,j}=\infty$\\
\hline
EC-QPT & $m_0=1$, $\sup_{s
} \frac{\sum_{j=1}^s b_j^{-1}}{1+\log s}< \infty$, and
        $\liminf_{j
} \frac{(1+\log j)\log a_j}{j}>0$ \\
\hline
EC-PT, EC-SPT & $m_0=1$, $\sum_{j=1}^{\infty} b_j^{-1} < \infty$, and
       $\liminf_{j
} \frac{\log a_j}{j}>0$\\
\hline
\end{tabular}
\end{center}

\section*{Acknowledgements}

The authors would like to thank two anonymous referees for suggestions on how to improve 
the presentation of the results.

\vspace{0.5cm}
\begin{small}
\noindent\textbf{Authors' addresses:}
\\ \\
\noindent Christian Irrgeher, Peter Kritzer, Friedrich Pillichshammer,
\\ 
Department of Financial Mathematics and Applied Number Theory, 
Johannes Kepler University Linz, Altenbergerstr.\ 69, 4040 Linz, Austria\\
 \\
\noindent Henryk Wo\'{z}niakowski, \\
Department of Computer Science, Columbia University, New York 10027,
USA, and Institute of Applied Mathematics, 
University of Warsaw, ul.\ Banacha 2, 02-097 Warszawa, Poland\\ \\

\noindent \textbf{E-mail:} \\
\texttt{christian.irrgeher@jku.at}\\
\texttt{peter.kritzer@jku.at}\\
\texttt{friedrich.pillichshammer@jku.at} \\
\texttt{henryk@cs.columbia.edu}
\end{small}

\end{document}